\documentclass[a4paper,10pt]{amsart}
\usepackage[utf8]{inputenc}
\usepackage{amsthm}
\usepackage{amsmath}
\usepackage{bm}
\usepackage{amsfonts}
\usepackage{amssymb}
\usepackage{mathtools}
\usepackage{mathrsfs}
\usepackage{setspace}
\usepackage{xcolor}
\usepackage{textgreek}
\usepackage{todonotes}
\usepackage[a4paper, left=3cm, right=3cm, top=3cm, bottom=3cm]{geometry}
\usepackage{thmtools} % solves problem of shared counters in statements and autoref
\usepackage{bm}

\usepackage[pdfdisplaydoctitle,colorlinks,breaklinks,urlcolor=blue,linkcolor=blue,citecolor=blue]{hyperref} %must always be the last

\newtheorem{thm}{Theorem}[section]

\newtheorem{definition}[thm]{Definition}
\newtheorem{cor}[thm]{Corollary}

\newtheorem{lem}[thm]{Lemma}

\newtheorem{prop}[thm]{Proposition}

\newcommand{\N}{\mathbb{N}}
\newcommand{\Z}{\mathbb{Z}}
\newcommand{\R}{\mathbb{R}}

\newcommand{\PP}{\mathbb{P}}
\newcommand{\T}{\mathbb{T}}

\newcommand{\dvg}{\mathord{{\rm div}}\,}

\theoremstyle{remark}
\newtheorem{rmk}[thm]{Remark}

\numberwithin{equation}{section}

\title[Anomalous and total dissipation due to advection by Navier-Stokes equations]{Anomalous and total dissipation due to advection by solutions of randomly forced Navier-Stokes equations}

\author{Martina Hofmanov\'a}
\address[M. Hofmanov\'a]{Fakult\"at f\"ur Mathematik, Universit\"at Bielefeld, D-33501 Bielefeld, Germany}
\email{hofmanova(at)math.uni-bielefeld.de}

\author{Umberto Pappalettera}
\address[U. Pappalettera]{Fakult\"at f\"ur Mathematik, Universit\"at Bielefeld, D-33501 Bielefeld, Germany}
\email{upappale(at)math.uni-bielefeld.de}

\author{Rongchan Zhu}
\address[R. Zhu]{Department of Mathematics, Beijing Institute of Technology, Beijing 100081, China}
\email{zhurongchan(at)126.com}

\author{Xiangchan Zhu}
\address[X. Zhu]{ Academy of Mathematics and Systems Science,
Chinese Academy of Sciences, Beijing 100190, China}
\email{zhuxiangchan(at)126.com}

\thanks{R.Z. and X.Z. are grateful for
the financial support by National Key R\&D Program of China (No. 2022YFA1006300).
R.Z. gratefully acknowledges financial support from the NSFC (No.  12271030) and BIT Science and Technology Innovation Program Project 2022CX01001.
X.Z. is grateful for the financial supports  in part by National Key R\&D Program of China (No. 2020YFA0712700) and the NSFC (No.   12090014, 12288201) and the support by key Lab of Random Complex Structures and Data Science, Youth Innovation Promotion Association (2020003), Chinese Academy of Science.
The research of  M.H. and U.P. was funded by the European Research Council (ERC) under the European Union's Horizon 2020 research and innovation programme (grant agreement No. 949981). The research of M.H., R.Z. and X.Z. was funded by the Deutsche Forschungsgemeinschaft (DFG, German Research Foundation) – Project-ID 317210226 – SFB 1283.
}

\date\today

\begin{document}

\begin{abstract}

We propose a novel approach to induce anomalous dissipation through advection driven by turbulent fluid flows.
Specifically, we establish the existence of a velocity field
$v$
satisfying randomly forced Navier-Stokes equations, leading to total dissipation of kinetic energy in finite time when advecting a passive scalar.  This dissipation phenomenon is uniform across viscosity parameters and initial conditions, representing a case of anomalous dissipation. We further explore dissipation induced by individual realizations of
$v$. Our results extend to scenarios where
the passive scalar
is replaced by solutions to two or three-dimensional deterministic Navier-Stokes equations advected by
$v$.

\end{abstract}

\maketitle

%%%%%%%%%%%%%%%%%%%%%%%%%%%%%%%%%%%%%%%%%%%%%%%%%%%%%%%%%%%%%%%%%%%%%%%%%%%%%%%%%%%%%%%%%%%%%

\tableofcontents

\section{Introduction}

We are concerned with the intriguing physical phenomenon of anomalous dissipation   observed in turbulent fluids.
This topic has been extensively discussed in the physics literature  \cite{O49, B59} in connection with passive scalar quantities $\rho$ like temperature or solute concentration  advected by the fluid velocity $v$. This scenario is mathematically modeled by the advection-diffusion equation:
\begin{align} \label{eq:passive_scalar_intro}
\partial_t \rho
+
v \cdot \nabla \rho
&=
\nu \Delta \rho.
\end{align}
We consider a situation where  the fluid is confined within a box with periodic boundaries,  denoted by the $d$-dimensional torus by $\T^d := (\R/\Z)^d$, where $d=2,3$. In equation \eqref{eq:passive_scalar_intro}, the unknown $\rho:\T^d \times \R_+ \to \R$ represents  the scalar quantity,  and $v : \T^d \times \R_+ \to \R^d$ represents the incompressible advecting velocity field, assumed to be predetermined and ideally governed by a physical model such as the Navier-Stokes or Euler equations.
The viscosity (or diffusivity) parameter $\nu$ is strictly positive. We assume that the initial condition $\rho_{0}$ in \eqref{eq:passive_scalar_intro} is non-constant, belongs to $L^{2}$ and is independent of $\nu$.

Anomalous dissipation refers to  the limiting behavior of the dissipation, specifically the time decay of the $L^{2}$-norm of $\rho=\rho^{\nu}$, as the viscosity approaches zero. More precisely, anomalous dissipation at time $t=1$  occurs if
\begin{align} \label{eq:anomalous2}
\liminf_{\nu \to 0}  \|\rho^{\nu}_1\|_{L^2} < \|\rho_0\|_{L^2},
\end{align}
or equivalently, if the energy equality holds,
\begin{align} \label{eq:anomalous}
\limsup_{\nu \to 0} \nu \int_0^1 \| \nabla \rho_t \|_{L^2}^2 dt > 0.
\end{align}
The latter  formulation is common in the physics literature. However, depending on the regularity of the advecting velocity field, the energy equality may only hold under additional assumptions on the initial condition and generally must be replaced by the corresponding energy inequality.

Physically, anomalous dissipation occurs due to the transfer of kinetic energy from $\rho$ to small spatial scales by the convective term $v \cdot \nabla \rho$. Mathematically, capturing this phenomenon is highly challenging, as demonstrated by the scarcity of rigorous mathematical results, as discussed in \autoref{sec:lit1} and \autoref{sec:lit2}.

Initially,  the claim \eqref{eq:anomalous2} may seem   counterintuitive. On the one hand, one  expects the solutions $\rho=\rho^{\nu}$ of \eqref{eq:passive_scalar_intro} to converge to a solution of the transport equation
\begin{align}\label{eq:transp}
\partial_t \rho
+
v \cdot \nabla \rho
= 0.
\end{align}
On the other hand, formally testing  \eqref{eq:passive_scalar_intro} by its solution leads to the energy equality
\begin{align*}%\label{eq:en}
\|\rho^{\nu}_{1}\|_{L^{2}}^{2}+2\nu\int_{0}^{1}\|\nabla\rho^{\nu}_{t}\|^{2}_{L^{2}}d t = \|\rho_0\|^{2}_{L^2}.
\end{align*}
Thus, if sufficient  uniform  in $\nu$ bound allowed for the passage to the limit in this expression,  the second term on the left hand side would  vanish, and
the limit solution would  conserve energy:
$$
\|\rho_{1}\|_{L^{2}}^{2}=\lim_{\nu\to 0}\|\rho^{\nu}_{1}\|_{L^{2}}^{2}= \|\rho_0\|^{2}_{L^2}.
$$

This is consistent with the direct (formal) testing of \eqref{eq:transp} by its solution. However, it contradicts anomalous dissipation \eqref{eq:anomalous2}, suggesting that anomalous dissipation is only valid for irregular solutions, which in turn implies that the advecting velocity cannot be excessively regular. Negative results are also known; for instance, for $L^{2}$ initial conditions $\rho_{0}$, \cite[Remark 4.3]{BoCiCi22} shows the absence of anomalous dissipation for vector fields $v\in L^{1}([0,1],W^{1,p})$ for some $p>1$ or $v\in L^{1}([0,1],W^{1,1})\cap L^{q}((0,1)\times\T^{d})$ for some $q>1$. This notably excludes the Leray-Hopf solutions $v$ to the Navier-Stokes equations, which belong to $L^{2}([0,1],H^{1})$ (see also \cite[Theorem~4]{DrElIyJe22}).

\subsection{Main results}

Taking into account the aforementioned counter-argument, we revisit the framework of Leray-Hopf solutions $v$ to the Navier-Stokes equations augmented with an additional stochastic force and a friction term. This exploration is partly motivated by \cite{MaKr99}, wherein the authors assert that ``{\em the turbulent velocity which advects the passive scalar should be a solution to the Navier-Stokes equations [...] with some external stirring which maintains the fluid in a turbulent state}". Our main result demonstrates the existence of a solution adhering to the Leray-Hopf regularity locally in time, and inducing enhanced and ultimately anomalous dissipation in \eqref{eq:passive_scalar_intro}.
In fact, we establish a significantly more robust result concerning total dissipation over finite time, uniformly across both the initial condition and the viscosity parameter.

\begin{thm} \label{thm:main_passive}
There exist a countable family of Brownian motions $\{W^{k,\alpha}\}$, time-dependent velocity fields $\{\sigma_{k,\alpha}\}$, and a weak solution $v$ of the Navier-Stokes equations with large friction and additive noise:
\begin{align} \label{eq:veps3}
\begin{cases}
d v
+
(v \cdot \nabla) v \,dt
+
\nabla p_v \,dt
=
\Delta v \,dt
-
\varepsilon^{-1} v \, dt
+
\varepsilon^{-1}
\sum_{k,\alpha} \sigma_{k,\alpha} dW^{k,\alpha},
\\
\dvg v = 0,
\end{cases}
\end{align}
where $\varepsilon=\varepsilon(t)$ depends on $t$ and is piecewise constant, such that
\begin{align*}
v \in C_w([0,1),L^2) \cap L^2_{loc}([0,1),H^1)
\quad
\mbox {almost surely},
\end{align*}
and every progressively measurable\footnote{By weak continuity of weak solutions to \eqref{eq:passive_scalar_intro},
	progressive measurability is equivalent to require that $\rho$ is \emph{adapted}, namely for every $t \in [0,1)$ the random variable $\rho_t$ is measurable with respect to
	the sigma algebra $\mathcal{F}_t$, where $(\Omega,\mathcal{F},\{\mathcal{F}_t\}_{t \geq 0},\PP)$ is a
	filtered probability space, satisfying the usual conditions, that supports the Brownian motions $\{W^{k,\alpha}\}$.
} weak solution of \eqref{eq:passive_scalar_intro}
manifests total dissipation at time $t=1$, i.e. for every zero-mean initial condition $\rho_0 \in L^2$ and $\nu \in (0,1)$ it holds
\begin{align*}
\lim_{t \uparrow 1}
\| \rho_t \|_{L^2} = 0 \quad \mbox{ almost surely}.
\end{align*}
\end{thm}

The solution $v$, as provided by the preceding theorem, experiences a loss of Leray-Hopf regularity precisely at time $t=1$. This phenomenon arises due to the increasing strength of the forcing noise and friction in \eqref{eq:veps3}, amplified as time approaches $1$ via $\varepsilon(t)\to 0$ as $t\to 1$. In essence, $v$ undergoes  a carefully controlled blow up at  $t=1$, a behavior that can be further quantified based on the choice of $\varepsilon$ and the characteristics of the noise.
The mechanism driving total (not merely enhanced) dissipation involves a forced transfer of the kinetic energy of $\rho$ to increasingly higher  wavenumbers as time $t$ converges to $1$. Notably, the obtained result is \emph{universal}, as it holds true for every initial condition in $L^{2}$ and every $\nu \in (0,1)$.
Specifically, the dissipation is anomalous in the sense that \eqref{eq:anomalous2} holds almost surely by continuity extension, along with  \eqref{eq:anomalous}  if $\rho$ adheres to the energy equality\footnote{For the supposed regularity of $v$, this condition holds when $\rho_0 \in L^3$  for $d=3$ and $\rho_{0}\in L^{2^{+}}$ for $d=2$ by \cite[ Theorem~3.3]{BCC24}.}.
Central to our proof is the crucial property that solutions of \eqref{eq:passive_scalar_intro}, with $\rho_0 \in L^2$ and velocity $v \in L^2_{loc}([0,1),H^1)$,  exist and are unique due to \cite{BoCiCi22}, and satisfy the corresponding energy inequality. Such solutions are readily obtained via Galerkin approximation.

As an intermediate auxiliary  step, we establish the corresponding statement for the scenario of white-in-time velocity $v$ formally expressed as:
$$
v=\sum_{k,\alpha}\sigma_{k,\alpha}\partial_{t}W^{k,\alpha}.
$$
This result, novel and interesting in its own right, serves as a foundational component for the proof of \autoref{thm:main_passive}.

It is noteworthy to mention  that one could theoretically consider a specific realization $v=v(\omega)$ of the random velocity field $v$ above and inquire whether the \emph{deterministic} vector field $v(\omega)$ induces total dissipation in \eqref{eq:passive_scalar_intro}. However, the challenge in this scenario lies in the fact that the full measure set of $\omega$'s for which $\lim_{t \uparrow 1}\| \rho_t \|_{L^2} = 0$ in \autoref{thm:main_passive} may potentially depend on the initial condition $\rho_0$ and the viscosity $\nu$. Consequently, the universality of dissipation is compromised, yielding a weaker result as stated below:

\begin{cor} \label{cor:passive}
For every countable sets of initial conditions $\mathscr{C} \subset L^2$ with zero mean and viscosities $\mathscr{V} \subset (0,1)$ there exists a deterministic vector field $v=v(\mathscr{C},\mathscr{V})$ such that $\lim_{t \uparrow 1} \|\rho_t\|_{L^2}=0$ for every weak solution $\rho$ of \eqref{eq:passive_scalar_intro} with initial condition $\rho_0 \in \mathscr{C}$ and viscosity $\nu \in \mathscr{V}$. More generally, given arbitrary probability measures $\PP_{\rho_0}$ on $L^2$ with zero mean and $\PP_{\nu}$ on $(0,1)$, there exists  $v=v(\PP_{\rho_0},\PP_\nu)$ inducing total dissipation at time $t=1$ in \eqref{eq:passive_scalar_intro} for $\PP_{\rho_0} \otimes \PP_\nu$ almost every $(\rho_0,\nu)$.
\end{cor}

In cases where a higher degree of regularity for $v$ is desired, it can be postulated to satisfy an Ornstein-Uhlenbeck equation instead:
\begin{align} \label{eq:veps2_intro}
d v
=
-
\varepsilon^{-1} v \, dt
+
\varepsilon^{-1}
\sum_{k,\alpha} \sigma_{k,\alpha} dW^{k,\alpha}.
\end{align}

\begin{thm} \label{thm:OU}
There exists $v \in C^{1/2^-}_{loc}([0,1),C^\infty)$ almost surely, satisfying \eqref{eq:veps2_intro} above, such that the same total dissipation for solutions of \eqref{eq:passive_scalar_intro} stated in \autoref{thm:main_passive} holds true.
\end{thm}

Our methodology exhibits  robustness that extends to  nonlinear equations  as well. For instance, let us consider the Navier-Stokes equations advected by $v$ in dimension $d=2$ or $3$:
\begin{align} \label{eq:NS_intro}
\begin{cases}
\partial_t u
+
(u \cdot \nabla) u
+
(v \cdot \nabla) u
+
\nabla p
=
\nu \Delta u,
\\
\dvg u = 0,
\end{cases}
\end{align}
where  $u: \T^d \times \R_+ \to \R^d$ represents the unknown velocity field, and the scalar pressure field $p: \T^d \times \R_+  \to \R$ is disregarded from  the analysis of \eqref{eq:NS_intro} by applying the Leray projector $\Pi$.
Here is the resultant theorem:

\begin{thm} \label{thm:main_NS_d=2}
For the same velocity field $v$ as in \autoref{thm:main_passive} or  \autoref{thm:OU}, every progressively measurable Leray-Hopf weak solution of \eqref{eq:NS_intro} with zero-mean, divergence-free initial condition $u_0 \in L^2$ and $\nu \in (0,1)$ satisfies $\lim_{t \uparrow 1} \|u_t\|_{L^2} = 0$ almost surely.
\end{thm}
The previous theorem holds true both in dimension $d=2$ and $d=3$; however, there is a technical difference between these two cases.
When $d=2$, by pathwise uniqueness of solutions to \eqref{eq:NS_intro}, given the velocity field $v$ we can always find a Leray-Hopf weak solution $u$ of \eqref{eq:NS_intro} that is progressively measurable with respect to the filtration generated by the Brownian motions $\{W^{k,\alpha}\}$, and the analogue of \autoref{cor:passive} holds true.
Moreover, energy equality gives the analogue of \eqref{eq:anomalous}
\begin{align} \label{eq:anomalous_u}
\limsup_{\nu \to 0} \nu \int_0^1 \| \nabla u_t \|_{L^2}^2 dt > 0 \quad
\mbox{ almost surely}.
\end{align}

In dimension $d=3$, the Navier-Stokes equations may admit multiple Leray-Hopf weak solutions,  and thus it is unknown whether \eqref{eq:NS_intro} possesses Leray-Hopf weak solutions adapted to $\{W^{k,\alpha}\}$ when the noise is specified  a priori. Consequently, a statement akin to that of \autoref{thm:main_NS_d=2} might be void.
In such instances, the existence of progressively measurable Leray-Hopf weak solutions is recovered subject to a possible change of the underlying probability space. This concept  is commonly referred to as probabilistically weak existence.
We obtain the following result:

\begin{thm} \label{thm:main_NS}
Let $d=3$.
Then there exists a countable family of time-dependent velocity fields $\{\sigma_{k,\alpha}\}$ with the following property.
For every countable sets of initial conditions $\mathscr{C} \subset L^2$ with null mean and divergence and viscosities $\mathscr{V} \subset (0,1)$ there exist a countable family of Brownian motions $\{W^{k,\alpha}\}$ and a vector field $v \in C_w([0,1),L^2) \cap L^2_{loc}([0,1),H^1)$ solution of \eqref{eq:veps3} such that for every initial condition $u_0 \in \mathscr{C}$ and viscosity $\nu \in \mathscr{V}$ there exists a progressively measurable Leray-Hopf solution $u$ of \eqref{eq:NS_intro} satisfying $\lim_{t \uparrow 1} \|u_t\|_{L^2} = 0$ almost surely.
\end{thm}

Contrary to the assertion \autoref{cor:passive}, if we fix the initial condition $u_0$ and viscosity $\nu$, and then consider a deterministic realization of $v$,  total dissipation cannot be universally established for every (deterministic) Leray-Hopf weak solution $u$.
The reason is that we are not able to say whether a given Leray-Hopf weak solution $u$ is the realization $u=\tilde{u}(\omega)$ of a progressively measurable Leray-Hopf weak solution $\tilde{u}$.
Nonetheless, we obtain  total dissipation at time $t=1$ for at least one Leray-Hopf weak solution:

\begin{cor} \label{cor:NS}
Let $d=3$. Then for every zero-mean initial condition $u_0 \in L^2$ with null divergence and viscosity $\nu \in (0,1)$ there exist a deterministic vector field $v=v(u_0,\nu)$ and a Leray-Hopf weak solution $u$ of \eqref{eq:NS_intro} satisfying $\lim_{t \uparrow 1} \|u_t\|_{L^2} = 0$.
\end{cor}

Finally, let us mention that for $u$ solution of \eqref{eq:NS_intro} in dimension $d=3$ we do not obtain \eqref{eq:anomalous_u}, since we do not know whether energy equality holds and we might only have an energy inequality.
Additional details will be given in \autoref{sec:white_noise} and \autoref{sec:smooth}.
Note that the presence of the nonlinearity $(u\cdot \nabla)u$ in \eqref{eq:NS_intro} is not the main source of difficulties in our proof, nor it is the ultimate reason why the total dissipation occurs. As in the passive scalar case, total dissipation is  rather induced by the advection term $v\cdot\nabla u$.

\subsection{Known deterministic results}\label{sec:lit1}

One of the most extensively studied examples of dissipation enhancing vector fields is that of alternating shear flows -- velocity fields $v$ characterized by  translation invariance along one time-dependent direction.
These flows have been recognized for their capacity to induce enhanced dissipation \cite{BCZ17,CZ20}. More recently,  \cite{DrElIyJe22} established a general criterion for anomalous dissipation based on an inverse interpolation inequality. Specifically, the authors proved that  for every $\alpha\in [0,1)$ there exists a velocity field $v\in C^{\infty}([0,1)\times \T^{d})\cap L^{1}([0,1],C^{\alpha}(\T^{d}))\cap L^{\infty}([0,1]\times \T^{d})$ that yields  anomalous dissipation for initial conditions in $H^{2}$ near harmonics.

A more intricate variant of alternating shear flows involves  rearranging chessboard-like initial data to develop small-scale spatial structures, thus facilitating anomalous dissipation, as presented in \cite{CoCrSo22}. This work demonstrated that for every $\alpha\in [0,1)$ and $p\in [2,\infty]$, there exists an initial datum $\rho_{0}\in C^{\infty}$ and a velocity field $v\in L^{p}([0,1],C^{\alpha}(\T^{2}))$ inducing anomalous dissipation. Similarly, in \cite{BCCDLS22}, a result on anomalous dissipation for the forced Navier-Stokes equations was established. Notably, in both \cite{DrElIyJe22} and \cite{CoCrSo22}, anomalous dissipation occurred at the final time $t=1$. This temporal limitation was addressed in \cite{ArVi23+} through a fractal homogenization approach. For every $\alpha\in (0,1/3)$, the existence of a vector field $v\in C([0,1],C^{\alpha}(\T^{d}))\cap C^{\alpha}([0,1],C( \T^{d}))$ was obtained, ensuring anomalous dissipation for every initial datum in $H^{1}$.

In the aforementioned works, the advecting velocity was constructed somewhat ad hoc, without a priori reference to any specific fluid dynamics model. Incorporating such a requirement imposes additional constraints, rendering the problem even more challenging. This challenge was effectively tackled in the recent work \cite{BuSzWu23+}, which appeared after the completion of the first version of our manuscript. Here, weak solutions to Euler equations inducing anomalous dissipation were identified. Specifically, the authors demonstrated that for every $\alpha\in (0,1/3)$, there exists a weak solution $v\in C^{\alpha}([0,1]\times\T^{3})$ to the Euler equations, ensuring anomalous dissipation in \eqref{eq:passive_scalar_intro} for all initial conditions in $H^{1}$.

Moreover, other mechanisms have been shown to produce anomalous dissipation for particular initial data in the Navier-Stokes equations with force, as documented in \cite{BrDL22,JeYo21,JeYo22}.

\subsection{Known stochastic results}\label{sec:lit2}

Kraichnan proposed white-in-time random flows as a model of synthetic scalar turbulence in \cite{Kra68}, which has since been extensively studied (see \cite{MaKr99} and references therein, or the more recent \cite{GeYa21+}). The Kraichnan model is characterized by a specific spatial structure and demonstrates the phenomenon of Lagrangian spontaneous stochasticity, observed in \cite{BGK98}. This phenomenon is intimately related to scalar anomalous dissipation via a fluctuation-dissipation relation. Under certain assumptions, \cite{DE17} proved  the equivalence between these two notions, thus establishing anomalous dissipation for the Kraichnan model.

Furthermore, more general white-in-time random flows have gained attention in recent years. Thanks to works such as \cite{FlGaLu21+,FlGaLu22}, it is now relatively well-understood in stochastic literature that a white-in-time velocity field $v$ can induce mixing and  dissipation enhancement in \eqref{eq:passive_scalar_intro}, especially when $v$ is concentrated solely at very high Fourier modes and each mode is excited with low intensity. This dissipation enhancement stems from the presence of a Stratonovich-to-It\=o corrector in the weak formulation of \eqref{eq:passive_scalar_intro}, which, under suitable geometric conditions on the spatial structure of the noise, acts as a large multiple of the Laplacian, facilitating the proof of smallness of the $H^{-1}$ norm of the solution $\rho$. It is crucial to note that while this \emph{transport noise} formally preserves energy, its algebraic reformulation with the Stratonovich-to-It\=o corrector provides a convenient framework for effectively quantifying energy transfer to high wavenumbers.

Additionally,  the transport noise can be engineered to produce arbitrarily strong dissipation enhancement on short time intervals, as discussed in \cite{FlGaLu22}. However, until now, it remained unclear whether this dissipation could be made anomalous, i.e., uniform in the vanishing viscosity limit $\nu\to0$.

Moreover, the literature, e.g., \cite{MaKr99}, has argued against white-in-time velocity fields being physically relevant. This poses additional challenges, as no Stratonovich-to-It\=o corrector appears in the weak formulation of \eqref{eq:passive_scalar_intro}, rendering previous considerations inapplicable. Nevertheless, if $v$ exhibits a similar spatial structure to the aforementioned transport noise, it is reasonable to expect the same energy transfer to high wavenumbers and consequent dissipation enhancement to occur, as discussed in \cite{Pa22}.

\subsection{Our contribution}\label{sec:lit3}

The primary contribution of the present paper lies in capturing the phenomenon of enhanced and anomalous dissipation without relying on the Stratonovich-to-It\=o corrector. The core idea is to define $v$ as a solution of a stochastic differential equation driven by a strong external random forcing and a strong friction. These two elements enable $v$ to mimic the behavior of a transport noise in the weak formulation of \eqref{eq:passive_scalar_intro}, despite $v$ retaining a positive decorrelation time. With sufficiently strong external forcing, we can even allow $v$ to satisfy the forced Navier-Stokes equations or other equations relevant to fluid dynamics. This is particularly significant as it distinguishes our approach from most existing examples of anomalous dissipation, with the exception of \cite{BuSzWu23+}.

To quantify the induced decay of the $H^{-1}$ norm of $\rho$, we add a small perturbation to the solution of \eqref{eq:passive_scalar_intro}, reintroducing temporal roughness into the equation that produced  the Stratonovich-to-It\=o corrector  in the  case of  transport noise. The perturbation, inspired by homogenization methods and ideas from \cite{DePa22+}, ultimately results in the convective term homogenizing to a negative definite operator, which coincides with the same Stratonovich-to-It\=o corrector observed with transport noise, plus small reminders. Through suitable tuning of the parameters defining the random forcing acting on $v$, we establish enhanced, and eventually anomalous, dissipation in \eqref{eq:passive_scalar_intro}.

In terms of regularity over the entire time domain $[0,1]$, our construction represents an intermediate result between the white-in-time examples mentioned above and the purely deterministic results from \cite{DrElIyJe22,CoCrSo22,ArVi23+, BuSzWu23+}. Specifically, we eliminate the need for a white-in-time velocity but encounter integrability issues near $t=1$, unlike the deterministic examples mentioned above. This  can be intuitively understood as follows: as $t\to1$, we approximate the transport noise, resulting in time regularity of only $C^{-1/2^{-}}$ uniformly in $[0,1]$. Nevertheless, the noise simultaneously shifts to higher and higher modes and eventually vanishes as $t\to1$. This can be employed in the case of the Ornstein-Uhlenbeck process \eqref{eq:veps2_intro} to trade space regularity for time integrability. More precisely, as shown in \autoref{rmk:OU}, it holds true that $v\in L^{\infty}([0,1];H^{-\epsilon})$ for a suitably chosen $\epsilon>0$.
The Navier-Stokes setting \eqref{eq:veps3} is more complex due to the nonlinearity, and this argument does not  apply.

However, we can draw a parallel between our construction and the relaxation enhancing flows described in \cite{CKRZ08}, whose suitable acceleration was also shown to produce total dissipation in \cite{R24} by choosing a time dependent coefficient in the spirit of our construction. Notably, \cite{R24} mentions that total dissipation is unattainable for velocity fields $v\in L^{\infty}([0,1]\times\T^{d})$ at any diffusivity $\nu>0$. Therefore, in our scenario, solutions of the forced Navier-Stokes equations blow up in a relaxation enhancing manner almost surely, meaning they explode in the right way to yield mixing and diffusion. A generic velocity field experiencing blow-up at $t=1$ would not be expected to exhibit such behavior.

In summary, we have identified a novel mechanism for inducing scalar anomalous dissipation that: i) is potent enough to achieve total dissipation in finite time; ii) is universal, as a single velocity field $v$ is effective for every initial condition in $L^{2}$ and every viscosity value; and iii) is robust enough to extend to the nonlinear case, such as the Navier-Stokes equations, and likely to many other models of fluid dynamics interest.

\subsection{Organization of the paper}
The paper is organized as follows.
In \autoref{sec:ideas}, we describe the structure of our noise and outline the main ideas of how to reintroduce the roughness to profit from the hidden Stratonovich-to-It\=o corrector.
\autoref{sec:white_noise} addresses the preliminary problem of white-in-time velocity $v$, presented in \autoref{prop:dissipation_transport}. This is crucial for our analysis since it allows us to isolate the difficulties arising from the Leray projection $\Pi$ in the expression of the Stratonovich-to-It\=o corrector in \autoref{ssec:strato} and the main strategy behind the proof of total dissipation in \autoref{ssec:diss}.
In \autoref{ssec:non-diss}, we also demonstrate  that in dimension $d=3$, there exist weak solutions of \eqref{eq:NS_intro} perturbed by transport noise that are not Leray-Hopf and do not dissipate their kinetic energy at time $t=1$. This is reasonable to expect, since non Leray-Hopf weak solutions do not need to satisfy the energy inequality, see \eqref{eq:energy_ineq}.
\autoref{sec:smooth} explores  the more realistic scenario of $v$ being a solution of the forced Navier-Stokes equations.
The primary focus of this section is to rigorously validate  the heuristic arguments presented in \autoref{ssec:rough}, particularly in the more difficult case of solutions to \eqref{eq:NS_intro}.
The proof of \autoref{thm:OU} when $v$ solves \eqref{eq:veps2_intro} descends easily from the Navier-Stokes case \eqref{eq:veps3}, see \autoref{rmk:OU}.

\section{Main ideas of the construction}\label{sec:ideas}

\subsection{Structure of the noise} \label{ssec:structure}

Let us focus on the case $d=3$ only. The case $d=2$ is readily covered in a similar fashion and we omit it for the sake of brevity.

We construct the noise in \eqref{eq:veps3} taking inspiration from \cite{Ga20} and \cite{FlLu21}.
Let $\Z^3_0 := \Z^3 \setminus \{\textbf{0}\}$ and let $\{\Lambda,-\Lambda\}$ be a partition of $\Z^3_0$.
Let us introduce a family $\{ B^{k,\alpha} \}_{k \in \Z^3_0, \alpha \in \{1,2\}}$ of independent  standard Brownian motions on a filtered probability space $(\Omega,\mathcal{F},\{\mathcal{F}_t\}_{t \geq 0},\PP)$ satisfying the usual conditions.
Define the complex-valued Brownian motions
\begin{align*}
W^{k,\alpha} := \begin{cases}
B^{k,\alpha} + i B^{-k,\alpha}, & \mbox{ if } k \in \Lambda,
\\
B^{k,\alpha} - i B^{-k,\alpha}, & \mbox{ if } k \in -\Lambda,
\end{cases}
\end{align*}
whose quadratic covariation is given by
$[W^{k,\alpha},W^{l,\beta}]_t=2t\delta_{k,-l}\delta_{\alpha,\beta}$.
For every $k \in \Lambda$ let $\{a_{k,1},a_{k,2}\}$ be orthonormal basis of $k^\perp$ such that $\{a_{k,1},a_{k,2},k/|k|\}$ is right-handed. Let $a_{k,\alpha} = a_{-k,\alpha}$ for every $k \in \Z^3_0$.
The coefficients $\{\sigma_{k,\alpha}\}_{k \in \Z^3_0, \alpha \in \{1,2\}}$ are given by
\begin{align*}
\sigma_{k,\alpha}(x,t) :=
\theta_k(t) a_{k,\alpha} e_k(x),
\qquad
e_k(x) := e^{2\pi i k \cdot x},
\end{align*}
for some suitable intensity coefficients $\theta_k : \R_+ \to \R$.
Notice that $\{a_{k,\alpha} e_k\}_{k,\alpha}$ is a complete orthonormal system of the space $H$ of zero-mean, divergence-free, square integrable velocity fields on the torus; in particular $\sigma_{k,\alpha}(\cdot,t)$ is $H$ valued for every $k,\alpha$ and $t$.
For $s \in \R$, denote $H^s$ the $H$-based Sobolev space of velocity fields on the torus with zero average and null divergence in the sense of distributions.

We introduce the dependence of $\theta_k$ on time in order to switch the noise on at higher and higher Fourier modes as time approaches $t=1$.
To rigorously describe this, let $\{\tau_q\}_{q \in \N}$ be a sequence of times such that $\tau_0=1$ and $\tau_q \to 0$ monotonically as $q \to \infty$.
On the time interval $(1-\tau_q,1-\tau_{q+1}]$ we choose the coefficients $\theta_k(t)=\theta_k^q$ as:
\begin{align*}
\theta_k^q &:= \mathbf{1}_{\{N_q \leq |k|\leq 2N_q\}},
\end{align*}
where $N_q$ is a sequence such that $N_q \to \infty$ sufficiently fast as $q \to \infty$.
We also define
\begin{align*}
\sum_{k \in \Z^3_0} (\theta_k^q)^2
=: \kappa_q
&\sim
N_q^3  \to \infty
\qquad
\mbox{ as } q \to \infty.
\end{align*}

\subsection{Reintroducing roughness at small time scales} \label{ssec:rough}
Let us describe the heuristics behind the proof of \autoref{thm:main_passive}, which is based on ideas from \cite{DePa22+}.
For the sake of simplicity, here we focus on the scalar case but we shall see that the same arguments apply to the Navier-Stokes case.
Fix $q \in \N$ and consider the time interval $(1-\tau_q,1-\tau_{q+1}]$. Suppose on this time interval $\varepsilon(t) \equiv \varepsilon_q$ and we can decompose a solution $v$ of \eqref{eq:veps3} as
\begin{align*}
v = \varepsilon^{-1/2}_q w + r,
\end{align*}
where $w$ is the noisy part of $v$, satisfying
\begin{align*}
dw
=
-\varepsilon^{-1}_q w dt
+
\varepsilon^{-1/2}_q \sum_{k,\alpha} \sigma_{k,\alpha} dW^{k,\alpha},
\end{align*}
and is renormalized such that $w$ is on average of order one, whereas $r$ is a remainder.
Then, if $\varepsilon_q \ll 1$ the dynamics \eqref{eq:passive_scalar_intro}
\begin{align*}
d \rho
=
-
\varepsilon_q^{-1/2} w \cdot \nabla \rho \,dt
-
r \cdot \nabla \rho \,dt
+
\nu \Delta \rho \,dt
\end{align*}
has a potentially large term $\varepsilon_q^{-1/2} w \cdot \nabla \rho$.
However, this potentially troublesome term can be cancelled out by adding correctors to $\rho$, following ideas from  \cite{DePa22+}.
Roughly speaking, these correctors are auxiliary stochastic processes, usually taking values in a space of distributions, that formally compensate large terms in the dynamics of $\rho$.

In our particular case, we consider the dynamics of the process $\rho - \varepsilon_q^{1/2} w \cdot \nabla \rho$ instead. It reads as
\begin{align*}
d \left( \rho - \varepsilon_q^{1/2}  w \cdot \nabla \rho \right)
=
- r \cdot \nabla \rho \,dt
+
\nu \Delta \rho \,dt
+
w \cdot \nabla (w \cdot \nabla \rho) \,dt
-
\sum_{k,\alpha} \sigma_{k,\alpha} \cdot \nabla \rho \, dW^{k,\alpha}
+
O(\varepsilon_q^{1/2}).
\end{align*}
Here we have the term $w \cdot \nabla (w \cdot \nabla \rho)$ which oscillates extremely fast with respect to time, and together with the quadratic interaction in $w$ this produces the extra dissipation. To see this, consider another corrector defined as $\frac{\varepsilon_q}{2} w \cdot \nabla (w \cdot \nabla \rho) $. It evolves according to
\begin{align*}
\frac{\varepsilon_q}{2}
d \left( w \cdot \nabla (w \cdot \nabla \rho) \right)
&=
- w \cdot \nabla (w \cdot \nabla \rho) \,dt
+
\sum_{k,\alpha} \sigma_{k,\alpha} \cdot \nabla (\sigma_{-k,\alpha} \cdot \nabla \rho) \,dt
+
O(\varepsilon_q^{1/2}),
\end{align*}
which compensates exactly the fast oscillations $w \cdot \nabla (w \cdot \nabla \rho)$ in the dynamics of $\rho - \varepsilon_q^{1/2}  w \cdot \nabla \rho$ and simultaneously makes appear the Stratonovich-to-It\=o corrector $\sum_{k,\alpha} \sigma_{k,\alpha} \cdot \nabla (\sigma_{-k,\alpha} \cdot \nabla \rho)$.

If the parameters $\sigma_{k,\alpha}$ are tuned properly, this procedure permits to prove smallness of a negative Sobolev norm of the process
$$\rho - \varepsilon_q^{1/2}  w \cdot \nabla \rho  + \frac{\varepsilon_q}{2}  w \cdot \nabla (w \cdot \nabla \rho) + V \sim \rho,$$
from which we deduce dissipation of the $L^2$ norm of $\rho$ via an inverse interpolation inequality.
Here $V$ is another auxiliary process, of average size proportional to $\varepsilon_q$, which helps to control the term $-r \cdot \nabla \rho$ in the dynamics of $\rho$.
For more details, we refer the reader to \autoref{sec:smooth}, where the rigorous argument is described.

\begin{rmk} \label{rmk:OU}
The same arguments apply a fortiori to $v$ solution of \eqref{eq:veps2_intro}. Indeed, in this case one has $r=0$ in the decomposition $v=\varepsilon_q^{-1/2}w+r$.
Furthermore, in this scenario we are able to trade space regularity for time integrability and  show for instance that $v\in L^{\infty}([0,1],H^{-\epsilon})$ almost surely for a suitably chosen $\epsilon>0$.  Specifically, the Fourier coefficients  $v_{k, \alpha}
:= \langle v, a_{k, \alpha} e_k \rangle$   of  $v$ satisfy
\[ d v_{k, \alpha} + \varepsilon^{- 1} v_{k, \alpha} d t  = \varepsilon^{- 1} \theta_k d W^{k,
   \alpha}  . \]
Thus, fixing  $k,\alpha$ determines $q=q(k)$ via $N_{q}\leq|k|\leq 2N_{q}$ together with the time interval $(1-\tau_{q},1-\tau_{q+1}]$ where the above right hand side is non-zero.
In other words, integrating  the above equation in time, its right hand side remains bounded in time in expectation even in the limit $t \rightarrow
1$. However, the bound depends on $k, \alpha$: testing the equation by $v_{k,
\alpha}$ leads to
\[ \frac{1}{2} | v_{k, \alpha} (t) |^2 + \varepsilon^{- 1} \int_0^t | v_{k,
   \alpha}|^2 d s =  \int_{1-\tau_{q}}^{t\wedge (1-\tau_{q+1})} \varepsilon_{q}^{- 1} \theta^{q}_k v_{k, \alpha} d W^{k, \alpha} +
   \frac{1}{2} \int_{1-\tau_{q}}^{t\wedge (1-\tau_{q+1})}  \varepsilon_{q}^{- 2} (\theta^{q}_k)^{2}  d s \]
and the It\=o correction term explodes as $k \rightarrow \infty$, i.e. $q
\rightarrow \infty$. Applying Burkholder-Davis-Gundy's inequality for the stochastic integral, we  deduce
\begin{align*}
\mathbb{E}\sup_{t\in[0,1]}\|v_{t}\|_{H^{-\epsilon}}^{2}&=  \mathbb{E}\sup_{t\in[0,1]} \sum_{k,\alpha} k^{-2\epsilon}| v_{k, \alpha} (t) |^2\lesssim  \sum_{k}   k^{-2\epsilon}\varepsilon_{q(k)}^{- 2}\lesssim 1
\end{align*}
provided $k^{-2\epsilon}\ll \varepsilon_{q(k)}^{2}$ in order to cancel the blow up coming from $\varepsilon_{q(k)}^{-2}$ and to make the sum converging. On the one hand, we have $k\sim N_{q(k)}$. On the other hand, it can be seen in \autoref{sec:smooth} that for the case of the Ornstein-Uhlenbeck process we may choose $\varepsilon_{q}\sim N_{q}^{-K}$ for some integer $K\in\mathbb{N}$. This implies the condition $N_{q}^{-\epsilon}\ll N_{q}^{-K}$ hence $\epsilon>K$.

\end{rmk}

Hereafter, in order to keep the discussion as concise as possible, we prefer to restrict ourselves to the vectorial case \eqref{eq:NS_intro}, which is technically more demanding than the scalar case \eqref{eq:passive_scalar_intro} due to the more complex form of the Stratonovich-to-It\=o corrector.
Also, in view of \autoref{rmk:OU} we will only consider $v$ solution of \eqref{eq:veps3}, being the case when $v$ solves \eqref{eq:veps2_intro} easier.
Major differences in the proofs will be detailed when necessary.

\section{Total dissipation by transport noise} \label{sec:white_noise}

In this section we shall focus on the  case of transport noise: $$v(x,t) = \sum_{k,\alpha} \sigma_{k,\alpha}(x,t) \partial_t W^{k,\alpha},$$ where the coefficients $\sigma_{k,\alpha}$ are as in \autoref{ssec:structure} and the white noise $\partial_t W^{k,\alpha}$ is meant in the Stratonovich sense.
This is a good starting point, propaedeutic to the study of \eqref{eq:NS_intro} advected by a more realistic velocity field solution $v$ solution of \eqref{eq:veps3}.
With this choice of $v$, equation \eqref{eq:NS_intro} reads as
\begin{align} \label{eq:NS_transport}
\begin{cases}
d u
+
(u \cdot \nabla) u \,dt
+
\sum_{k,\alpha} (\sigma_{k,\alpha} \cdot \nabla) u
\circ dW^{k,\alpha}
+
\nabla p \,dt
=
\nu \Delta u \,dt
\\
\dvg u = 0.
\end{cases}
\end{align}

The symbol $\circ\, dW^{k,\alpha}$ in the equation above is a notational shorthand for the Stratonovich interpretation of the stochastic integral. This is a sensible modelling choice since Stratonovich integration satisfies the chain rule and thus the advection term $\sum_{k,\alpha} (\sigma_{k,\alpha} \cdot \nabla) u \circ dW^{k,\alpha}$ is formally energy preserving.

The goal of this section is to prove the following:
\begin{prop} \label{prop:dissipation_transport}
We can choose parameters $\tau_q$, $N_q$ and $\kappa_q \sim N_q^3$, $q \in \N$, such that the following holds true.
Every probabilistically weak, progressively measurable Leray-Hopf weak solution of $\eqref{eq:NS_transport}$ with zero-mean, divergence-free initial condition $u_0 \in L^2$ and viscosity $\nu \in (0,1)$ satisfies $\lim_{t \uparrow 1} \|u_t\|_{L^2} = 0$ almost surely.
\end{prop}

A similar statement holds true in the passive scalar case.
Before moving on, let us clarify what we mean by probabilistically weak Leray-Hopf weak solution of \eqref{eq:NS_transport} in this case.

\begin{definition} \label{def:sol_u}
A probabilistically weak, progressively measurable Leray-Hopf weak solution of \eqref{eq:NS_transport} is defined as a probability space $(\Omega,\mathcal{F},\{\mathcal{F}_t\}_{t\geq 1}, \PP)$ supporting a family of i.i.d. Brownian motions $\{W^{k,\alpha}\}_{k,\alpha}$ and a progressively measurable stochastic process $u : \Omega \to C_w([0,1),H) \cap L^2([0,1),H^1)$ almost surely such that, for every divergence-free test function $f \in C_c^\infty(\T^3 \times [0,1),\R^3)$ it holds almost surely for every $0\leq s<r < 1$
\begin{align*}
\langle u_r , f_r \rangle - \langle u_s , f_s \rangle
=
\int_s^r \langle u_t , \partial_t f_t + (u_t \cdot \nabla) f_t + \nu \Delta f_t\rangle dt
+
\sum_{k,\alpha}
\int_s^r \langle u_t , (\sigma_{k,\alpha}(\cdot,t) \cdot \nabla) f_t \rangle \circ dW^{k,\alpha}_t,
\end{align*}
and for almost every $\omega \in \Omega$ there exists a full Lebesgue measure set $\mathcal{T} \subset [0,1)$ such that $0 \in \mathcal{T}$ and for every $r \in \mathcal{T}$, $r<t<1$ the following energy inequality holds almost surely
\begin{align} \label{eq:energy_ineq}
\|u_t\|_{L^2}^2
+
2\nu\int_r^t \|\nabla u_s\|_{L^2}^2 ds
\leq
\|u_r\|_{L^2}^2.
\end{align}
\end{definition}

We can rewrite the Stratonovich integral in \autoref{def:sol_u} in the equivalent It\=o form as follows: for every divergence-free test function $f \in C_c^\infty(\T^3 \times [0,1),\R^3)$
\begin{align*}
\langle u_r , f_r \rangle - \langle u_s , f_s \rangle
&=
\int_s^r \langle u_t , \partial_t f_t + (u_t \cdot \nabla) f_t + \nu \Delta f_t + S(f_t)\rangle dt
\\
&\quad+
\sum_{k,\alpha}
\int_s^r \langle u_t , (\sigma_{k,\alpha}(\cdot,t) \cdot \nabla) f_t \rangle dW^{k,\alpha}_t,
\end{align*}
where $S$ is the so-called Stratonovich-to-It\=o (or simply Stratonovich) corrector, described in details in the next subsection. Existence of solutions according to \autoref{def:sol_u} has been shown in \cite{FlGa95}.

Finally, let us comment briefly on the energy inequality \eqref{eq:energy_ineq}.
It ultimately is the reason why the transfer of energy to high wavenumbers increases the rate of dissipation in solutions of the Navier-Stokes equations, since a larger time integral of $\| \nabla u_s \|_{L^2}^2$ necessitates a smaller $\|u_t\|_{L^2}^2$ to ensure the validity of \eqref{eq:energy_ineq}.
However, the energy inequality can not be deduced from the equation \eqref{eq:NS_transport} itself, and must be postulated a priori: this is the difference between weak solutions and Leray-Hopf weak solutions of \eqref{eq:NS_transport}.
In particular, weak solutions of \eqref{eq:NS_transport} that do not satisfy the energy inequality do not need to dissipate their energy, and explicit examples have been constructed recently by convex integration techniques \cite{BuVi19}, \cite{BuMoSz20}, \cite{Pa23+}.
In \autoref{ssec:non-diss} we give an example following the construction of \cite{Pa23+}.

In the passive scalar case, the weak formulation of the equation is obtained \emph{mutatis mutandis}, but by linearity and pathwise uniqueness \cite[Theorem 5.2]{Fl95} the stronger energy equality holds true: one can take $\mathcal{T}=[0,1)$ and has almost surely for every $t>r$
\begin{align*}
\|\rho_t\|_{L^2}^2
+
2\nu\int_r^t \|\nabla \rho_s\|_{L^2}^2 ds
=
\|\rho_r\|_{L^2}^2.
\end{align*}
This can be obtained for instance taking a space mollification $\rho^\epsilon$ of the solution, using a commutator estimate à la Di Perna-Lions \cite[Lemma II.1]{DPLi89} to deduce energy equality for $\rho^\epsilon$ up to in infinitesimal error, and then passing to the limit $\epsilon \to 0$. The passage to the limit is justified by local-in-time smoothness of the coefficients $\{\sigma_{k,\alpha}\}$.

\subsection{Stratonovich corrector} \label{ssec:strato}
Equation \eqref{eq:NS_transport} can be rewritten without the pressure term by using the Leray projector $\Pi=Id+\nabla (-\Delta)^{-1} \dvg$ onto the divergence-free velocity fields:
\begin{align} \label{eq:NS_transport_Leray}
d u
+
\Pi [(u \cdot \nabla) u] \,dt
+
\sum_{k,\alpha} \Pi [(\sigma_{k,\alpha} \cdot \nabla) u ] \circ dW^{k,\alpha}
=
\nu \Delta u \,dt.
\end{align}
The advantage of \eqref{eq:NS_transport_Leray} compared to \eqref{eq:NS_transport} is that the Stratonovich-to-It\=o corrector $S$ can be readily computed,\footnote{The problem with formulation \eqref{eq:NS_transport} is related to the difficulty of computing the quadratic covariation between the Brownian motions $W^{k,\alpha}$ and the pressure $p$.} see also equation (1.6) in \cite{FlLu21}. We have
\begin{align*}
d u
+
\Pi [(u \cdot \nabla) u] \,dt
+
\sum_{k,\alpha} \Pi [(\sigma_{k,\alpha} \cdot \nabla) u ]  \,dW^{k,\alpha}
=
\nu \Delta u \,dt+S(u) \,dt,
\end{align*}
with
\begin{align} \label{eq:defin_strat}
S(u) =
\sum_{k,\alpha}
\Pi[(\sigma_{k,\alpha} \cdot \nabla) \Pi[(\sigma_{-k,\alpha} \cdot \nabla) u]].
\end{align}

On the time interval $(1-\tau_q,1-\tau_{q+1}]$, we can rewrite $S=S_q$  (cf. (2.6) and (2.7) and choose $C_\nu$ properly therein) as
\begin{align*}
S_q(u) &= \frac23 \kappa_q \Delta u - S_q^\perp (u),
\\
S_q^\perp(u) &:=
\sum_{k,\alpha}
(\theta_k^q)^2\,\Pi[(a_{k,\alpha}e_k \cdot \nabla) \Pi^\perp(a_{k,\alpha}e_{-k} \cdot \nabla) u],
\end{align*}
where $\Pi^{\perp} = - \nabla (-\Delta)^{-1} \dvg$ is the orthogonal complement of the Leray projector $\Pi$.

Decompose $u$ as $u = \sum_{\ell,\beta} u_{\ell,\beta} a_{\ell,\beta} e_\ell$. By \cite[Corollary 5.3]{FlLu21} it holds
\begin{align*}
S_q^\perp(u) &= -4\pi^2
\sum_{\ell,\beta} u_{\ell,\beta}|\ell|^2
\Pi \left[\sum_k
(\theta_k^q)^2 \sin^2(<_{k,\ell})
(a_{\ell,\beta} \cdot (k-\ell) ) \frac{k-\ell}{|k-\ell|^2}e_\ell
\right],
\end{align*}
where $<_{k,\ell}$ denotes the angle between the $\Z^3$ vectors  $k$ and $\ell$.

In order to better describe the behaviour of $S_q^\perp$, split $u$ into low and high Fourier modes
\begin{align*}
u = u^L + u^H,
\quad
u^L := \Pi_L u,
\end{align*}
where $\Pi_L$ denotes the Fourier projector onto modes ${|\ell|\leq N_q^{1-\delta}}$ for some small $\delta>0$. It holds
\begin{align*}
S_q^\perp(u^L)
&=
-4\pi^2
\sum_{\ell,\beta} u^L_{\ell,\beta}|\ell|^2
\Pi \left[\sum_k
(\theta_k^q)^2 \sin^2(<_{k,\ell})
(a_{\ell,\beta} \cdot k ) \frac{k}{|k|^2}e_\ell
\right]
\\
&\quad-4\pi^2
\sum_{\ell,\beta} u^L_{\ell,\beta}|\ell|^2
\Pi \left[\sum_k
(\theta_k^q)^2 \sin^2(<_{k,\ell})
\left((a_{\ell,\beta} \cdot (k-\ell) ) \frac{k-\ell}{|k-\ell|^2} -
(a_{\ell,\beta} \cdot k ) \frac{k}{|k|^2} \right)e_\ell
\right].
\end{align*}
Next we show that the second line on the right-hand-side in the expression above is negligible when compared to the first line, at least when taking the $H^{-1}$ norm thereof.
Recall also that $\Pi$ is a bounded operator from $H^{-1}$ to $H^{-1}$ and acts diagonally on the Fourier elements $e_\ell$.
By \cite[Lemma 5.5]{FlLu21} using $u^L_{\ell,\beta}=0$ for $|\ell| >N_q^{1-\delta}$ and $\theta_k^q = 0$ for $|k| < N_q$ it holds
\begin{align*}
&\left\|\sum_{\ell,\beta} u^L_{\ell,\beta}|\ell|^2
\Pi \left[\sum_k
(\theta_k^q)^2 \sin^2(<_{k,\ell})
\left((a_{\ell,\beta} \cdot (k-\ell) ) \frac{k-\ell}{|k-\ell|^2} -
(a_{\ell,\beta} \cdot k ) \frac{k}{|k|^2} \right)e_\ell
\right]
\right\|_{H^{-1}}
\\
&\quad \leq
\left(
\sum_{\ell,\beta} |u^L_{\ell,\beta}|^2|\ell|^2
\left( \sum_k
(\theta_k^q)^2
\left|(a_{\ell,\beta} \cdot (k-\ell) ) \frac{k-\ell}{|k-\ell|^2} -
(a_{\ell,\beta} \cdot k ) \frac{k}{|k|^2} \right| \right)^2
\right)^{1/2}
\\
&\quad \lesssim
\left(
\sum_{\ell,\beta} |u^L_{\ell,\beta}|^2|\ell|^2
\left(\sum_k
(\theta_k^q)^2
\frac{|\ell|}{N_q}\right)^2
\right)^{1/2}
\lesssim
\|u^L\|_{H^1}\kappa_q N_q^{-\delta}.
\end{align*}

On the other hand, one can check that the convergences in \cite[Proposition 5.4, Lemma 5.6]{FlLu21} hold true with $\gamma=0$ and are uniform in $|\ell| \leq N_q^{1-\delta}$, namely we have
\begin{align*}
&\left\|\sum_{\ell,\beta} u^L_{\ell,\beta}|\ell|^2
\Pi \left[\sum_k
(\theta_k^q)^2 \left( \sin^2(<_{k,\ell})
(a_{\ell,\beta} \cdot k ) \frac{k}{|k|^2}
-\frac{4}{15} a_{\ell,\beta}
\right)e_\ell
\right]
\right\|_{H^{-1}}
%\\
%&\quad
\lesssim
\|u^L\|_{H^1}\kappa_q N_q^{-1}.
\end{align*}

Therefore, since
\begin{align*}
-4\pi^2 \sum_{\ell,\beta} u^L_{\ell,\beta}|\ell|^2 \sum_k
(\theta_k^q)^2 a_{\ell,\beta} e_\ell
=
\kappa_q \Delta u^L,
\end{align*}
we have for some universal constant $C$
\begin{align} \label{eq:bound_strat}
\left\|\frac25 \kappa_q \Delta u^L - S_q(u^L)  \right\|_{H^{-1}}
=
\left\| S_q^\perp(u^L) - \frac{4}{15} \kappa_q\Delta u^L \right\|_{H^{-1}}
\leq
C\kappa_q N_q^{-\delta} \|u^L\|_{H^1}.
\end{align}
A similar computation shows the continuity of $S_q : \Pi_L H^{s+2} \to H^{s}$ for other values of $s \in \R$.

\begin{rmk}
In the passive scalar case, there is no need to reformulate the equation as in \eqref{eq:NS_transport_Leray} since no Leray projection is needed. As a consequence, the Stratonovich-to-It\=o corrector takes the simpler form
\begin{align*}
S_q(\rho)
=
\sum_{k,\alpha} (\theta_k^q)^2 a_{k,\alpha}e_k \cdot \nabla (a_{k,\alpha}e_{-k} \cdot \nabla \rho)
=
\frac{2}{3} \kappa_q \Delta \rho.
\end{align*}
\end{rmk}

\subsection{Dissipation} \label{ssec:diss}
In principle, we are interested in the case where the initial condition $u_0 \in L^2$ is deterministic. However, without any additional difficulty we can assume $u_0$ random and independent of the Brownian motions $\{W^{k,\alpha}\}$, and that $u_0 \in L^2$ holds almost surely.
Indeed, this setting is equivalent to requiring $\|u_0\|_{L^2} \leq M$ almost surely, for some deterministic constant $1 \leq M <\infty$.
To see this, introduce $\Omega_M := \{ M-1 < \|u_0\|_{L^2} \leq M\}$ for $M \in \N$, $M \geq 1$ and write $u_0 = \sum_M u_0 \mathbf{1}_{\Omega_M}$. Denoting $u^M$ a Leray-Hopf weak solution of \eqref{eq:NS_transport} with initial condition $u_0 \mathbf{1}_{\Omega_M}$, then $u:= \sum_M u^M$ solves \eqref{eq:NS_transport} with initial condition $u_0$. Viceversa, every Leray-Hopf weak solution to \eqref{eq:NS_transport} can be decomposed as above.

Let $q \in \N$ be given. A key step to the proof of total dissipation is the following uniform estimate on the $H^{-1}$ norm of $u^L$, the low-modes projection of $u$. This implies smallness of the $H^{-1}$ norm of $u$ since $u^H$ is high-modes by construction, and thus the standard estimate $\|u^H\|_{H^{-1}} \lesssim N_q^{\delta-1}\|u^H\|_{L^2} \leq N_q^{\delta-1}\|u\|_{L^2}$ holds true (recall that by definition of $H$ the process $u$ has zero space average for every $t$ almost surely).

\begin{lem} \label{lem:h-1}
For every $q \in \N$ there exists a choice of the parameters $\tau_q$, $N_q$ and $\kappa_q \sim N_q^3$ such that the following holds true.
Let $u$ be any Leray-Hopf weak solution of \eqref{eq:NS_transport} with initial condition $u_0$, $\|u_0\|_{L^2} \leq M$.
Then it holds for every $\beta>3/2$ and $\delta \in (0, 1 \wedge(\beta-3/2))$
\begin{align*}
\mathbb{E}\left[\sup_{t \in [1-\tau_q/2,1-\tau_{q+1}]}\|u^L\|_{H^{-1}}^2\right] \lesssim_{\beta,\delta}
\frac{M^{2\frac{\beta+1}{\beta}}}{(\nu+\kappa_q)^{\frac{1-\delta}{\beta}}}.
\end{align*}
\end{lem}

\begin{proof}
Let us decompose $u=u^L + u^H$ and write down the mild formulation for $u^L$ with $t \in (1-\tau_q,1-\tau_{q+1}]$
\begin{align*}
u^L_t
&=
P(t-(1-\tau_q)) u^L_{1-\tau_q}
-
\int_{1-\tau_q}^t
P(t-s)
\Pi_L \Pi\, \dvg(u_s \otimes u_s) ds
\\
&\quad-
\sum_{k,\alpha}\int_{1-\tau_q}^t
P(t-s)
\Pi_L \Pi\, \dvg(u_s \otimes \sigma_{k,\alpha}) dW^{k,\alpha}_s
\\
&\quad+
\int_{1-\tau_q}^t
P(t-s)\left( S_q(u^L_s) -\frac25 \kappa_q \Delta u^L_s \right)
 ds.
\end{align*}
In the expression above we have conveniently denoted $P=e^{(\nu + 2\kappa_q/5)\Delta}$ the semigroup generated by the operator $(\nu + 2\kappa_q/5)\Delta$, and we have used that $\Pi_L S_q = S_q \Pi_L$ since $S_q$ acts diagonally on the Fourier elements $e_\ell$.

Since $u$ is a Leray-Hopf solution, the energy inequality guarantees $\|u^L_t\|_{L^2} \leq \|u_t\|_{L^2} \leq M$ almost surely for every $t$.
Thus using $\| P(t) \|_{H^{-1} \to H^{-1}} \lesssim \frac{1}{(\nu+2\kappa_q/5)t}$ we have
\begin{align*}
\mathbb{E}\left[ \sup_{t \in [1-\tau_q/2,1-\tau_{q+1}]}
\|P(t-(1-\tau_q)) u^L_{1-\tau_q}\|_{H^{-1}}^2
\right]
\lesssim
\frac{M^2}{(\nu+\kappa_q)^2\tau_q^2},
\end{align*}
and by boundedness of $\Pi_L\Pi$ and the embedding $L^1 \subset H^{-\beta+\delta}$
\begin{align*}
\mathbb{E}&\left[ \sup_{t \in [1-\tau_q/2,1-\tau_{q+1}]}
\int_{1-\tau_q}^t
\|P(t-s)
\Pi_L \Pi\,\dvg(u_s \otimes u_s) \|_{H^{-\beta}}^2 ds
\right]
%\\
\\
&\lesssim
\mathbb{E}\left[ \sup_{t \in [1-\tau_q/2,1-\tau_{q+1}]}
\int_{1-\tau_q}^t \frac{\| u_s \otimes u_s \|_{H^{-\beta+\delta}}^2}{(\nu+\kappa_q)^{1-\delta}(t-s)^{1-\delta}}
 ds
\right]
\\
&\lesssim
\mathbb{E}\left[ \sup_{t \in [1-\tau_q/2,1-\tau_{q+1}]}
\int_{1-\tau_q}^t \frac{\| u_s \otimes u_s \|_{L^1}^2}{(\nu+\kappa_q)^{1-\delta}(t-s)^{1-\delta}}
 ds
\right]
\lesssim
\frac{M^4}{(\nu+\kappa_q)^{1-\delta}}.
\end{align*}
In principle, the parameter $\delta$ in the equation above does not need to coincide with the parameter $\delta$ defining the low modes projector $\Pi_L$; however, we have chosen to use the same value for both  quantities, so to avoid the introduction of an additional parameter and keep notation lighter.

To control the stochastic integral, we want to apply \cite[Lemma 2.5]{FlGaLu21+}, in particular equation (2.3) therein. Notice that, with respect to that lemma, here we have in addition the projector $\Pi_L \Pi$; but it is easy to check that the very same proof applies also in this case. Therefore
\begin{align*}
\mathbb{E}&\left[
\sup_{t \in [1-\tau_q/2,1-\tau_{q+1}]} \left\|
\sum_{k,\alpha}\int_{1-\tau_q}^t
P(t-s)
\Pi_L\Pi\, \dvg(u_s \otimes \sigma_{k,\alpha}) dW^{k,\alpha}_s
\right\|_{H^{-\beta}}^2
\right]
\lesssim
\frac{M^2}{(\nu + \kappa_q)^{1-\delta}}.
\end{align*}

Finally, for the last term we have by \eqref{eq:bound_strat} and using that $P$ and $S_q$ commute
\begin{align*}
&\left\|\int_{1-\tau_q}^t
P(t-s)\left(S_{q}(u_{s}^L) -\frac25 \kappa_q \Delta u_{s}^L  \right)ds\right\|_{H^{-1}}
\lesssim\kappa_qN_q^{-\delta}\int_{1-\tau_q}^t
\|P(t-s)u^L_{s}\|_{H^1} ds
\\&\quad\lesssim\kappa_qN_q^{-\delta}\int_{1-\tau_q}^t
\|P(t-s)u^L_{s}\|_{H^{2-2\varepsilon}} ds
\lesssim\kappa_qN_q^{-\delta}\int_{1-\tau_q}^t
(\nu+\kappa_q)^{-1+\varepsilon}(t-s)^{-1+\varepsilon}\|u^L_{s}\|_{L^2} ds
\\&\quad\lesssim M\kappa_q^\varepsilon N_q^{-\delta},
\end{align*}
where $\varepsilon>0$ is small enough. Accordingly,
\begin{align*}
\mathbb{E}&\left[\sup_{t \in [1-\tau_q/2,1-\tau_{q+1}]}
\left\| \int_{1-\tau_q}^t
P(t-s)\left( S_q(u^L_s) -\frac25 \kappa_q \Delta u^L_s \right)
ds\right\|_{H^{-1}}^2
\right]\lesssim
M^{2}\kappa_{q}^{2\varepsilon}N_{q}^{-2\delta}.
\end{align*}
Putting all together and assuming for every $q \in \N$
\begin{align} \label{eq:condition_parameters}
\kappa_q^{2\varepsilon} N_q^{-2\delta}
\lesssim
\frac{1}{(\nu+\kappa_q)^{1-\delta}},
\qquad
(\nu+\kappa_q)^{1+\delta} \tau_q^2 \gtrsim 1,
\end{align}
by interpolation we arrive to
\begin{align*}
\mathbb{E}\left[\sup_{t \in [1-\tau_q/2,1-\tau_{q+1}]}\|u^L\|_{H^{-1}}^2\right]
&\leq
M^{2\frac{\beta-1}{\beta}} \,
\mathbb{E}\left[\sup_{t \in [1-\tau_q/2,1-\tau_{q+1}]}\|u^L\|_{H^{-\beta}}^{2}\right]^{1/\beta}
\lesssim
\frac{M^{2\frac{\beta+1}{\beta}}}{(\nu+\kappa_q)^{\frac{1-\delta}{\beta}}}
.
\end{align*}

\end{proof}

With this lemma at hand we are ready to prove our dissipation result in the case of transport noise.
\begin{proof}[Proof of \autoref{prop:dissipation_transport}]
Let \eqref{eq:condition_parameters} holds true.
Without loss of generality we can assume $\|u_0\|_{L^2} \leq M$ almost surely, for some deterministic constant $M \in (0,\infty)$.
For simplicity we denote
\begin{align*}
\tilde{c}_q^2 := C_{\beta,\delta}  \frac{M^{2/\beta}}{(\nu+\kappa_q)^{\frac{1-\delta}{\beta}}},
\end{align*}
where $C_{\beta,\delta}$ is the implicit constant in the previous lemma and depends only on $\beta$ and $\delta$.
The same lemma implies, by Markov inequality:
\begin{align*}
\|u_t\|_{H^{-1}}^2
&\leq
\|u_t^L\|_{H^{-1}}^2
+
\|u_t^H\|_{H^{-1}}^2
\leq
(\tilde{c}_q + N_q^{2(\delta-1)})M^2
=:
c_qM^2
\end{align*}
for every $t \in [1-\tau_q/2,1-\tau_{q+1}]$, with probability at least equal to $1-\tilde{c}_q$.

Recall the energy inequality \eqref{eq:energy_ineq} satisfied by Leray-Hopf weak solutions of \eqref{eq:NS_transport}: there exists a full Lebesgue measure set of times $\mathcal{T}=\mathcal{T}(\omega)$ such that for every $r \in \mathcal{T}$, $t>r$ it holds
\begin{align*}
\|u_t\|_{L^2}^2
+
2\nu\int_r^t \|u_s\|_{H^1}^2 ds
\leq
\|u_r\|_{L^2}^2.
\end{align*}

Formally, in order to prove smallness of $\|u_{1-\tau_{q+1}}\|_{L^2}$ we would like to apply the following inequality with $c_q \ll 1$:
\begin{align} \label{eq:d/dt_energy}
-\frac{d}{dt} \left(\frac{1}{\|u_t\|_{L^2}^2} \right)
=
\frac{1}{\|u_t\|_{L^2}^4 }\frac{d}{dt} \|u_t\|_{L^2}^2
&\leq
-2\nu \frac{\|u_t\|_{H^1}^2}{\|u_t\|_{L^2}^4 }
\leq
-\frac{2\nu }{\|u_t\|_{H^{-1}}^2}
\leq
-\frac{2 \nu}{c_q M^2},
\end{align}
where the second to last inequality comes from interpolation $\|u_t\|_{L^2}^4 \leq \|u_t\|_{H^1}^2 \|u_t\|_{H^{-1}}^2$.
However, the kinetic energy $t \mapsto \|u_t\|_{L^2}^2$ may have jump discontinuities and even vanish at some time $t \in [0,1)$.
In order to rigorously make sense of the previous line, let us consider instead the function
\begin{align*}
E(t) :=
\begin{cases}
\|u_t\|_{L^2}^2, &\mbox{ if } t \in \mathcal{T},
\\
\displaystyle
\liminf_{\substack{s \in \mathcal{T},\, s \to t}}
\|u_s\|_{L^2}^2, &\mbox{ if } t \in \mathcal{T}^c.
\end{cases}
\end{align*}
The function $E$ is non-increasing by definition of $\mathcal{T}$, therefore of class $BV$ almost surely.
In particular, $\liminf_{s \in \mathcal{T}, s \to t}\|u_s\|_{L^2}^2$ always equals the right limit $\lim_{s \in \mathcal{T}, s \downarrow t}\|u_s\|_{L^2}^2$ and $E$ is right continuous.
Moreover, also the map $t \mapsto \|u_t\|_{L^2}^2$ is of class $BV$ almost surely since it coincides with $E$ on the full Lebesgue measure set $\mathcal{T}$.
In particular, it holds as Radon measures
$$
d\|u\|_{L^2}^2((s,t]) = dE ((s,t]) = E(t)-E(s).
$$

Finally, $E(t) \geq \|u_t\|_{L^2}^2$ for every $t \in [0,1)$ since
\begin{align*}
\liminf_{\substack{s \in \mathcal{T},\, s \to t}}
\|u_s\|_{L^2}^2
\geq
\liminf_{\substack{s \to t}}
\|u_s\|_{L^2}^2
\geq
\|u_t\|_{L^2}^2
\end{align*}
by lower semicontinuity of the $L^2$ norm (recall that we require Leray-Hopf weak solutions to be of class $C_w([0,1),H)$ almost surely).

Let us consider the Lebesgue decomposition of the measure $dE$:
\begin{align*}
dE(t) = e_{Leb}(t) dt + d E_{Can}(t) + \sum_{t \in \mathcal{D}} (E(t^+)-E(t^-)) \delta_t,
\end{align*}
where $e_{Leb}(t)$ is the density with respect to the Lebesgue measure of the absolutely continuous part of $dE$, $d E_{Can}$ in the Cantor part of $dE$, $\mathcal{D}$ is the (at most countable) discontinuity set of $E$ and $\delta_t$ denotes the delta Dirac measure at time $t$.

By energy inequality \eqref{eq:energy_ineq} we have for almost every $t \in [0,1)$
\begin{align*}
e_{Leb}(t) \leq -2\nu \|u_t\|_{H^1}^2,
\end{align*}
and the Cantor and atomic parts of $dE$ are non-positive measures:
\begin{align*}
d E_{Can}(t) \leq 0,
\quad
\sum_{t \in \mathcal{D}} (E(t^+)-E(t^-)) \delta_t \leq 0.
\end{align*}

Recall that we want to show that $\|u_{1-\tau_{q+1}}\|_{L^2}^2$ is small. Assume $E(1-\tau_{q+1})>0$ (otherwise there is nothing to prove); by Vol'pert formula (i.e. the chain rule for $BV$ functions, see for instance \cite{AmFuPa00}) we have with probability no less than $1-\tilde{c}_q$
\begin{align*}
-
\frac{1}{E(1-\tau_{q+1})}
&\leq
\frac{1}{E(1-\tau_q/2)}
-
\frac{1}{E(1-\tau_{q+1})}
\\
&=
\int_{1-\tau_q/2}^{1-\tau_{q+1}}
\frac{e_{Leb}(t) dt}{E(t)^2}
+
\int_{1-\tau_q/2}^{1-\tau_{q+1}}
\frac{dE_{Can}(t)}{E(t)^2}
\\
&\quad+
\sum_{t \in \mathcal{D}}
\mathbf{1}_{\{1-\tau_q/2 <t\leq 1-\tau_{q+1}\}}
\left(-\frac{1}{E(t^+)} + \frac{1}{E(t^-)}
\right)
\\
&\leq
-2\nu \int_{1-\tau_q/2}^{1-\tau_{q+1}}
\frac{\|u_t\|_{H^1}^2 dt}{\|u_t\|_{L^2}^4 }
\leq
-2\nu \int_{1-\tau_q/2}^{1-\tau_{q+1}}
\frac{dt}{\|u_t\|_{H^{-1}}^2 }
\leq
-\frac{\nu (\tau_q-2\tau_{q+1})}{c_q M^2}.
\end{align*}

In the last line we have used $\|u_t\|_{L^2}^2 = E(t)$ for every $t \in \mathcal{T}$ and $\|u_t\|_{H^{-1}}^2 \leq c_q M^2$ for every $t \in [1-\tau_q/2,1-\tau_{q+1}]$ with probability at least $1-\tilde{c}_q$.

Therefore, as long as $q \in \N$ is such that $E(1-\tau_{q+1})>0$ and assuming $\tau_q-2\tau_{q+1} \geq \tau_q/2$, the previous formula implies
\begin{align*}
\mathbb{P} (A_q)
\geq 1-\tilde{c}_q,
\quad
A_q := \left\{
E(1-\tau_{q+1})
\leq
M^2 \frac{2 c_q}{\nu \tau_q}
\right\}.
\end{align*}

Let us now fix parameters $\beta=12/5$, $\delta=4/5$ and
\begin{align*}
\tau_q := 4^{-q},
\quad
N_q := \tau_q^{-10}.
\end{align*}
Since $\sum_{q}\tilde{c}_q<\infty$ by our choice of parameters, by Borel-Cantelli Lemma almost every $\omega \in \Omega$ belongs to the set $A_q$ for every $q$ larger than a certain $q_\star=q_\star(\omega)$, and thus
\begin{align*}
\sup_{t \geq 1-\tau_{q+1}} \|u_t\|_{L^2}^2
\leq
E(1-\tau_{q+1}) \to 0
\end{align*}
almost surely as $q\to \infty$, since for every fixed value of $M$ and $\nu$
\begin{align*}
\lim_{q \to \infty}
M^2 \frac{ c_q}{\nu \tau_q}
=
0.
\end{align*}
This obviously implies $\lim_{t\uparrow 1} \|u_t\|_{L^2}=0$ almost surely, and the proof is complete.
\end{proof}

\begin{rmk} \label{rmk:rate}
As a consequence, Leray-Hopf weak solutions of \eqref{eq:NS_transport} can be extended with $u_1=0$ to continuous functions at time $t=1$ with respect to the strong topology on $H$.

Actually, from the proof of the previous proposition one can deduce the following refinement. Recall that the largest integer $q$ such that $E(1-\tau_{q+1})>M^2 \frac{c_q}{\nu \tau_q}$ is almost surely finite by Borel-Cantelli Lemma.
In particular, we have the following almost sure ``rate'' of dissipation:
\begin{align*}
\limsup_{q \to \infty} \left(\frac{\nu \tau_q}{M^2 c_q}\right)^{1-\delta}E(1-\tau_{q+1})
\leq
\lim_{q \to \infty} \left(\frac{M^2 c_q}{\nu \tau_q}\right)^{\delta} =0,
\end{align*}
where $\delta \in (0,1)$, implying
\begin{align} \label{eq:rate}
\limsup_{q \to \infty} \left(\frac{\nu \tau_q}{M^2 c_q}\right)^{1-\delta} \sup_{t \geq 1-\tau_{q+1}}\|u_{t}\|_{L^2}^2
=
0.
\end{align}
\end{rmk}

\subsection{Non dissipating solutions} \label{ssec:non-diss}

The key property that allowed us to prove total dissipation at time $t=1$ in \autoref{prop:dissipation_transport} was the energy inequality \eqref{eq:energy_ineq} satisfied by Leray-Hopf weak solutions of \eqref{eq:NS_transport}.
More than that, \eqref{eq:rate} shows that in this case some sort of \emph{enhanced dissipation} holds even before time $t=1$, although for $t<1$ it is neither total nor anomalous.
In this subsection we show that weak solutions not satisfying the energy inequality may not dissipate energy close to time $t=1$.

\begin{prop} \label{prop:conv_int}
Let $d=3$.
For every $\nu>0$ and zero-mean, divergence-free $u_0 \in L^2$ almost surely there exists a progressively measurable weak solution $u$ to \eqref{eq:NS_transport} on the time interval $[0,1)$, with continuous trajectories in $H^{-1}$ and initial condition $u|_{t=0}=u_0$, such that almost surely
\begin{align} \label{eq:non_diss}
\int_{r}^1 \|u_t\|_{L^2}^2 dt = \infty,
\quad
\forall r \in (0,1).
\end{align}
In particular, \eqref{eq:rate} can not hold true for the constructed solution.
\end{prop}

The proof is based on a modification of the convex integration scheme of \cite{Pa23+}. Let $Z$ be the unique weak solution on $[0,1)$ of the Stokes system
\begin{align*}
\begin{cases}
d Z  + \sum_{k,\alpha} (\sigma_{k,\alpha} \cdot \nabla) Z \circ dW^{k,\alpha} + \nabla p_Z dt = \nu\Delta Z dt,
\\
\dvg Z = 0,
\\
Z|_{t=0} = u_0.
\end{cases}
\end{align*}
It is sufficient to prove the following:
\begin{lem}
For every $K>0$ and $q \in \N$ there exists a progressively measurable weak solution $u^q$ to \eqref{eq:NS_transport} on the time interval $[1-\tau_q,1-\tau_{q+1}]$ with continuous trajectory in $H^{-1}$ and such that $u^q_{1-\tau_q} = Z_{1-\tau_q}$ and $u^q_{1-\tau_{q+1}} = Z_{1-\tau_{q+1}}$ almost surely, and with probability at least $1-(q+1)^{-2}$ it holds
\begin{align*}
\int_{1-\tau_q}^{1-\tau_{q+1}} \|u^q_t\|_{L^2}^2 dt \geq K.
\end{align*}
\end{lem}

\begin{proof}
Without loss of generality we can assume $\|u_0\|_{L^2}\leq M$ almost surely.
The idea is to construct a solution $u^q$ as the sum of the solution $Z$ to the Stokes system and a perturbation $v$, under the additional constraint $v_{1-\tau_q}=v_{1-\tau_{q+1}}=0$.
In order to do so, we modify the convex integration scheme in \cite[Proposition 4.2]{Pa23+} using two sided cutoffs $\chi$ such that
\begin{align*}
\chi(t)
=
\begin{cases}
0, \quad \mbox{ if }
-\infty < t \leq 1-\tau_q(1 - 2^{-n-1}),
\\
1, \quad \mbox{ if }
1-\tau_q(1 - 2^{-n}) \leq t \leq 1-\tau_{q+1}(1+2^{-n}),
\\
0,\quad \mbox{ if }
1-\tau_{q+1}(1+2^{-n-1}) \leq t < \infty,
\end{cases}
\end{align*}
and monotone in between of these intervals.
Notice that imposing the terminal value $v_{1-\tau_{q+1}}=0$ does not compromise adaptedness, since the time $1-\tau_{q+1}$ is deterministic.
Then the estimates in \cite[Section 4.2]{Pa23+} remain the same, with the only differences that the the iterative estimates can now depend on $\nu>0$ and time derivative of $\chi$ is now controlled with $|\chi'| \lesssim \tau_{q+1}^{-1} 2^n$ and therefore may depend on $q$; however this gives no additional problem since here $\nu,q$ are fixed (notice however that the constants $C_v,C_R,...$ in \cite[Proposition 4.2]{Pa23+} may depend on $\nu$ and $\tau_{q+1}$).
The lower bound on the kinetic energy of $u^q$ comes from (here $u^q_n$ denotes the solution of the Navier-Stokes-Reynolds system obtained as $n$-th iteration of the convex integration scheme of \cite[Proposition 4.2]{Pa23+})
\begin{align*}
\int_{1-\tau_q(1 - 2^{-n+2}) \wedge \mathfrak{t}}^\mathfrak{t}
| \| u^q_{n+1}(t)-Z(t)\|_{L^2}^2 -\| u^q_n (t)-Z(t)\|_{L^2}^2 - 3 \gamma_{n+1} |\, dt
&\leq
C_e \delta_{n+1},
\end{align*}
where $\delta_n$ is a given parameter going to zero sufficiently fast as $n \to \infty$, $\mathfrak{t} \leq 1-\tau_{q+1}$ is a suitable stopping time with $\PP\{\mathfrak{t}>1-\tau_{q}/2\} \geq 1-(q+1)^{-2}$, and $C_e$ may depend on $\nu$ and $\tau_{q+1}$.
Then, choosing $n_0$ sufficiently large (possibly depending on $\nu$ and $q$) and $\gamma_{n_0} \geq C \tau_{q}^{-1} (K+M^2)$, we get with probability no less than $1-(q+1)^{-2}$:
\begin{align*}
2\int_{1-\tau_q}^{1-\tau_{q+1}} \|u^q_t\|_{L^2}^2 dt
&\geq
\int_{1-\tau_q(1 - 2^{-n+2}) \wedge \mathfrak{t}}^\mathfrak{t} \|u^q_t-Z(t)\|_{L^2}^2 dt - 2M^2
\\
&\geq
\int_{1-\tau_q(1 - 2^{-n+2}) \wedge \mathfrak{t}}^\mathfrak{t} \|u^q_{n_0}(t)-Z(t)\|_{L^2}^2 dt - C_e \sum_{n=n_0}^\infty \delta_{n+1}- 3 \sum_{n = n_0}^\infty \gamma_{n+1} -2M^2
\\
&\geq
\int_{1-\tau_q(1 - 2^{-n+2}) \wedge \mathfrak{t}}^\mathfrak{t} 3 \gamma_{n_0} dt - C_e \sum_{n} \delta_{n+1} - 3 \sum_{n \neq n_0} \gamma_n  - 2M^2
\geq
K.
\end{align*}
\end{proof}

\begin{proof}[Proof of \autoref{prop:conv_int}]
It follows from the previous lemma by gluing solutions $u^q$ on different time intervals and Borel-Cantelli Lemma. Gluing is possible because $Z_{1-\tau_{q+1}} \in L^2$ almost surely, and the glued process solves \eqref{eq:NS_transport} and enjoys continuity in $H^{-1}$.
Indeed for every divergence-free test function $f \in C^\infty_c(\T^3 \times [0,1), \R^3)$ and $0<s < 1-\tau_{q} < ...< 1-\tau_{q+k} < t<1$, since $u$ is a solution on every time interval $[1-\tau_{q+i-1},1-\tau_{q+i}]$ (the endpoints $s$ and $t$ being similar) it holds
\begin{align*}
\langle u_{1-\tau_{q+i}}^- , f_{1-\tau_{q+i}} \rangle - \langle u_{1-\tau_{q+i-1}}^+ , f_{1-\tau_{q+i-1}} \rangle
&=
\int_{1-\tau_{q+i-1}}^{1-\tau_{q+i}}  \langle u_r , \partial_t f_r + (u_r \cdot \nabla) f_r + \nu \Delta f_r \rangle dr
\\
&\quad+
\sum_{k,\alpha}
\int_{1-\tau_{q+i-1}}^{1-\tau_{q+i}}  \langle u_r ,  (\sigma_{k,\alpha} \cdot \nabla) f_r \rangle \circ dW^{k,\alpha},
\end{align*}
where $u_{1-\tau_{q+i}}^-$ denotes the left limit of $u$ at time $1-\tau_{q+i}$ and $u_{1-\tau_{q+i-1}}^+$ denotes the right limit of $u$ at time $1-\tau_{q+i-1}$.
By continuity in $H^{-1}$
\begin{align*}
\langle u_{1-\tau_{q+i}}^- , f_{1-\tau_{q+i}} \rangle - \langle u_{1-\tau_{q+i-1}}^+ , f_{1-\tau_{q+i-1}} \rangle
=
\langle u_{1-\tau_{q+i}} , f_{1-\tau_{q+i}} \rangle - \langle u_{1-\tau_{q+i-1}} , f_{1-\tau_{q+i-1}} \rangle,
\end{align*}
and therefore
\begin{align*}
\langle u_t , f_t \rangle - \langle u_s , f_s \rangle
&=
\langle u_t , f_t \rangle - \langle u_{1-\tau_{q+k}} , f_{1-\tau_{q+k}} \rangle
+
...
+
\langle u_{1-\tau_{q}} , f_{1-\tau_{q}} \rangle - \langle u_s , f_s \rangle
\\
&=
\int_s^t  \langle u_r , \partial_t f_r + (u_r \cdot \nabla) f_r + \nu \Delta f_r \rangle dr
+
\sum_{k,\alpha}
\int_s^t  \langle u_r ,  (\sigma_{k,\alpha} \cdot \nabla) f_r \rangle \circ dW^{k,\alpha}.
\end{align*}
\end{proof}

\section{Total dissipation by solution of randomly forced Navier-Stokes equations} \label{sec:smooth}
In this section we consider a relatively more regular (in time) approximation of transport noise.
Let the coefficients $\{\sigma_{k,\alpha}\}_{k,\alpha}$ be as in the previous section, and consider the Navier-Stokes equations with large friction and additive noise \eqref{eq:veps3}:
\begin{align*}
\begin{cases}
d v
+
(v \cdot \nabla) v \,dt
+
\nabla p_v \,dt
=
\Delta v \,dt
-
\varepsilon^{-1} v \, dt
+
\varepsilon^{-1}
\sum_{k,\alpha} \sigma_{k,\alpha} dW^{k,\alpha},
\\
\dvg v = 0,
\end{cases}
\end{align*}
where $\varepsilon=\varepsilon(t)$ depends on $t$ and is constantly equal to $\varepsilon_q \in (0,1)$ on the intervals of the form $(1-\tau_q,1-\tau_{q+1}]$. We shall assume $\varepsilon_q \to 0$ sufficiently fast as $q\to\infty$.

\begin{definition} \label{def:sol_v}
Given a probability space $(\Omega,\mathcal{F},\{\mathcal{F}_t\}_{t\geq 1}, \PP)$ supporting a family of i.i.d. Brownian motions $\{W^{k,\alpha}\}_{k,\alpha}$, a weak solution of \eqref{eq:veps3} is defined as a progressively measurable stochastic processes $v:\Omega \to C_w([0,1),H) \cap L^2_{loc}([0,1),H^1)$ almost surely such that, for every divergence-free test function $f \in C^\infty_c(\T^3 \times [0,1),\R^3)$ it holds almost surely for every $0\leq s<r < 1$
\begin{align*}
\langle v_r , f_r \rangle - \langle v_s , f_s \rangle
&=
\int_s^r \langle v_t , \partial_t f_t + (v_t \cdot \nabla) f_t +
\Delta f_t\rangle dt
\\
&\quad-
\int_s^r \varepsilon_t^{-1} \langle v_t,f_t\rangle dt
+
\sum_{k,\alpha}
\int_s^r \varepsilon_t^{-1} \langle \sigma_{k,\alpha}(\cdot,t),f_t  \rangle dW^{k,\alpha}_t.
\end{align*}
\end{definition}

Our main results are about total dissipation for progressively measurable (Leray-Hopf) weak solutions. For the sake of completeness, here we specify our notion of solutions.

\begin{definition} \label{def:sol_rho_bis}
Given a probability space $(\Omega,\mathcal{F},\{\mathcal{F}_t\}_{t\geq 1}, \PP)$ supporting a family of i.i.d. Brownian motions $\{W^{k,\alpha}\}_{k,\alpha}$ and a
%progressively measurable stochastic process $v:\Omega \to %C_w([0,1),H) \cap L^2_{loc}([0,1),H^1)$
weak solution $v$ of \eqref{eq:veps3}, a progressively measurable weak solution of \eqref{eq:passive_scalar_intro} is defined as a progressively measurable stochastic process $\rho : \Omega \to C_w([0,1),H) \cap L^2([0,1),H^1)$ almost surely such that, for every test function $f \in C^\infty_c(\T^3 \times [0,1),\R^3)$ it holds almost surely for every $0\leq s<r < 1$
\begin{align*}
\langle \rho_r , f_r \rangle - \langle \rho_s , f_s \rangle
=
\int_s^r \langle \rho_t , \partial_t f_t  +
\nu \Delta f_t\rangle dt
+
\int_s^r \langle \rho_t , (v_t \cdot \nabla) f_t \rangle dt,
\end{align*}
and for almost every $\omega \in \Omega$ there exists a full Lebesgue measure set $\mathcal{T} \subset [0,1)$ such that $0 \in \mathcal{T}$ and for every $r \in \mathcal{T}$, $t>r$ the following energy inequality holds almost surely
\begin{align*}
\|\rho_t\|_{L^2}^2
+
2\nu\int_r^t \|\nabla \rho_s\|_{L^2}^2 ds
\leq
\|\rho_r\|_{L^2}^2.
\end{align*}
\end{definition}

\begin{definition} \label{def:sol_u_bis}
Given a probability space $(\Omega,\mathcal{F},\{\mathcal{F}_t\}_{t\geq 1}, \PP)$ supporting a family of i.i.d. Brownian motions $\{W^{k,\alpha}\}_{k,\alpha}$ and a
%progressively measurable stochastic process $v:\Omega \to %C_w([0,1),H) \cap L^2_{loc}([0,1),H^1)$
weak solution $v$ of \eqref{eq:veps3}, a progressively measurable Leray-Hopf weak solution of \eqref{eq:NS_intro} is defined as a progressively measurable stochastic process $u : \Omega \to C_w([0,1),H) \cap L^2([0,1),H^1)$ almost surely such that, for every divergence-free test function $f \in C^\infty_c(\T^3 \times [0,1),\R^3)$ it holds almost surely for every $0\leq s<r < 1$
\begin{align*}
\langle u_r , f_r \rangle - \langle u_s , f_s \rangle
=
\int_s^r \langle u_t , \partial_t f_t + (u_t \cdot \nabla) f_t +
\nu \Delta f_t\rangle dt
+
\int_s^r \langle u_t , (v_t \cdot \nabla) f_t \rangle dt,
\end{align*}
and for almost every $\omega \in \Omega$ there exists a full Lebesgue measure set $\mathcal{T} \subset [0,1)$ such that $0 \in \mathcal{T}$ and for every $r \in \mathcal{T}$, $t>r$ the following energy inequality holds almost surely
\begin{align*}
\|u_t\|_{L^2}^2
+
2\nu\int_r^t \|\nabla u_s\|_{L^2}^2 ds
\leq
\|u_r\|_{L^2}^2.
\end{align*}
\end{definition}

Let us comment briefly on the previous notions of solutions \autoref{def:sol_rho_bis} and \autoref{def:sol_u_bis}.

Suppose $(\Omega,\mathcal{F},\{\mathcal{F}_t\}_{t\geq 1}, \PP)$, $\{W^{k,\alpha}\}_{k,\alpha}$ and $v$ are given.
For the passive scalar case \autoref{def:sol_rho_bis}, since we have assumed $v \in C_w([0,1),H) \cap L^2_{loc}([0,1),H^1)$ almost surely and $\rho_0 \in L^2$ there exists a unique weak solution $\rho$, which satisfies the energy inequality.
This can be shown following the lines of \cite[Corollary II.1]{DPLi89}. By uniqueness, the restriction $\rho|_{[0,t]}$ of $\rho$ to a time interval $[0,t]$, $t<1$ coincides, up to time $t$, with the solution obtained from the same initial condition and advecting velocity $v|_{[0,t]}$. In particular, $\rho$ is necessarily adapted to the filtration $\{\mathcal{F}_t\}_{t \geq 1}$ and therefore progressively measurable.
The same applies to the Navier-Stokes case \autoref{def:sol_u_bis} in dimension $d=2$.
However, in dimension $d=3$ we can not prove uniqueness of solutions, and the restriction $u|_{[0,t]}$ of a solution may depend on the values of $v$ after time $t$. For instance, it could be that the converging subsequence obtained by compactness depends on the whole trajectory of $v$.
As a consequence, there is no guarantee that $u$ is progressively measurable, and in order to regain adaptedness we may need to change the underlying probability space, see for instance \cite{FlGa95}.

The proof of \autoref{thm:main_passive}, \autoref{thm:main_NS_d=2} and \autoref{thm:main_NS} is based on a generator approach inspired by \cite{DePa22+}, which roughly speaking permits us to mimic the proof of \autoref{lem:h-1} even though technically speaking there is no Stratonovich corrector in the equations \eqref{eq:passive_scalar_intro} and \eqref{eq:NS_intro}, since now $v$ has positive decorrelation time and therefore $\rho$ and $u$ are processes with finite variation.

Similar arguments can be applied without additional difficulties to the Ornstein-Ulhenbeck approximation \eqref{eq:veps2_intro}, but we shall omit details for the sake of brevity.

Let us consider \eqref{eq:veps3} and let us split $v$ on the time interval $[1-\tau_q,1-\tau_{q+1}]$ as
\begin{align*}
v = \varepsilon^{-1/2}_q w + r,
\end{align*}
where $w(1-\tau_q) = 0$, $r(1-\tau_q)= v(1-\tau_q) \in H$ almost surely, and $w,r$ are divergence-free and evolve according to
\begin{align*}
dw
=
-\varepsilon^{-1}_q w dt
+
\varepsilon^{-1/2}_q Q^{1/2}_q dW,
\end{align*}
\begin{align*}
dr
=
-\varepsilon^{-1}_q r dt
+
A(\varepsilon^{-1/2}_q w + r) dt
+
b(\varepsilon^{-1/2}_q w + r,\varepsilon^{-1/2}_q w + r)dt.
\end{align*}
In the lines above we have denoted for simplicity
\begin{align*}
Q_q := \sum_{k,\alpha} (\theta_k^q)^2 ( a_{k,\alpha} e_k \otimes a_{k,\alpha} e_{-k} ),
\qquad
W := \sum_{k,\alpha} a_{k,\alpha} e_k W^{k,\alpha},
\end{align*}
so that $Tr(Q_q) = \kappa_q$ and
\begin{align*}
Q_q^{1/2}dW
=
\sum_{k,\alpha} \theta_k^q a_{k,\alpha} e_k dW^{k,\alpha}
=
\sum_{k,\alpha} \sigma_{k,\alpha} dW^{k,\alpha},
\end{align*}
and the operators $A$ and $b$ are defined as
\begin{align*}
A := \Delta,
\qquad
b(v_1,v_2) := -\Pi[(v_1 \cdot \nabla) v_2].
\end{align*}

Next we are going to collect some energy-type a priori estimates on the processes $v$, $w$, and $r$.
Then our total dissipation results hold true as soon as $v$ is a weak solution to \eqref{eq:veps3} that can be decomposed as $v = \varepsilon^{-1/2}_q w + r$ on each interval of the form $[1-\tau_q,1-\tau_{q+1}]$, these energy estimates hold true, and $\rho$ (resp. $u$) is a progressively measurable weak solution to \eqref{eq:passive_scalar_intro} (resp. Leray-Hopf weak solution to \eqref{eq:NS_intro}).

We point out that it is easy to exhibit at least one process $v$ satisfying this property, for instance by taking the limit of the Galerkin approximations on $(1-\tau_q,1-\tau_{q+1}]$
\begin{align*}
v^n = \varepsilon^{-1/2}_q w^n + r^n,
\qquad
n \in \N,
\end{align*}
to produce a (Leray-Hopf) weak solution $v$ to \eqref{eq:veps3} on $[1-\tau_q,1-\tau_{q+1}]$, living in a probability space $(\Omega^q,\mathcal{F}^q,\{\mathcal{F}^q_t\}_{t \geq 0},\PP^q)$ supporting the Brownian motions $\{W^{k,\alpha}\}_{k,\alpha}$ with $N_q \leq |k| \leq 2N_q$,
and then using continuity of $v$ e.g. in $H^{-1}$ to glue together solutions on different time intervals, see also \cite{FlRo08}.

It is worth mentioning that since each Brownian motion $W^{k,\alpha}$ has non-zero intensity $\theta_k$ for at most one time interval of the form $(1-\tau_q,1-\tau_{q+1}]$, the probability space $(\Omega,\mathcal{F},\{\mathcal{F}_t\}_{t \geq 0},\PP)$ supporting the whole family $\{W^{k,\alpha}\}_{k,\alpha}$ can be just taken as the product of the probability spaces $(\Omega^q,\mathcal{F}^q,\{\mathcal{F}^q_t\}_{t \geq 0},\PP^q)$.

Notice that, given $(\Omega,\mathcal{F},\{\mathcal{F}_t\}_{t \geq 0},\PP)$, $\{W^{k,\alpha}\}_{k,\alpha}$ and $v$ as above, there always exists a progressively measurable weak solution $\rho$ to \eqref{eq:passive_scalar_intro} and (in dimension $d=2$ only) $u$ to \eqref{eq:NS_intro}. This is a consequence of probabilistically weak existence and pathwise uniqueness, by Yamada-Watanabe Theorem.

However, when $d=3$ we are not able to construct Leray-Hopf weak solutions to \eqref{eq:NS_intro} that are adapted to the filtration $\{\mathcal{F}_t\}_{t \geq 0}$, and thus we need to define the probability space $(\Omega,\mathcal{F},\{\mathcal{F}_t\}_{t \geq 0},\PP)$ and Brownian motions $\{W^{k,\alpha}\}_{k,\alpha}$ taking into account the adaptedness of $u$, too.
This can be done considering simultaneously the Galerkin approximations of $v$ and $u$, working at fixed divergence-free initial condition $u_0 \in L^2$ and viscosity $\nu \in (0,1)$. More generally, one can fix countable families $\mathscr{C} \subset L^2$ with null divergence and $\mathscr{V} \subset (0,1)$ and consider simultaneously the Galerkin approximations of $v$ and $u^{u_0,\nu}$, where $u^{u_0,\nu}$ solves \eqref{eq:NS_intro} with initial condition $u_0 \in \mathscr{C}$ and viscosity $\nu \in \mathscr{V}$.
This marks the difference between the statements of \autoref{thm:main_passive}, \autoref{thm:main_NS_d=2} and the statement of \autoref{thm:main_NS}.

\subsubsection{Estimates on $v$}
The basic a priori estimate on $v_t$ we get from \eqref{eq:veps3}, which in particular holds true on the Galerkin approximations $v^n$ for every $t \in (1-\tau_{q},1-\tau_{q+1}]$, is the following
\begin{align*}
\mathbb{E}
\|v^n_{t} \|_{L^2}^2
&+
2 \int_{1-\tau_{q}}^t \mathbb{E}\|v^n_s \|_{H^1}^2 ds
+
2 \varepsilon^{-1}_q \int_{1-\tau_{q}}^t \mathbb{E}\|v^n_s \|_{L^2}^2 ds
\\
&\leq \nonumber
\mathbb{E}\|\Pi_n v_{1-\tau_{q}} \|_{L^2}^2
+
\varepsilon^{-2}_q \kappa_q (\tau_q-\tau_{q+1}),
\end{align*}
where in the right-hand-side $\Pi_n$ is the Fourier projector on modes $|k| \leq n$, and $v_{1-\tau_{q}}$ is considered as a given initial condtion (we suppose to have already defined the solution $v$              for times $t \leq 1-\tau_q$).
It is obtained by applying the It\=o formula to $\|v^n_t\|^2_{L^2}$ and taking expectations.

Therefore we deduce the following energy estimate on $v$:
\begin{align*}
\sup_{t \in (1-\tau_{q},1-\tau_{q+1}]}\mathbb{E}
\|v_{t} \|_{L^2}^2
&\leq
\mathbb{E}
\|v_{1-\tau_{q}} \|_{L^2}^2
+
\varepsilon^{-2}_q \kappa_q (\tau_q-\tau_{q+1})
\\
&\leq
\sup_{t \in (1-\tau_{q-1},1-\tau_{q}]}\mathbb{E}
\|v_{t} \|_{L^2}^2
+
\varepsilon^{-2}_q \kappa_q (\tau_q-\tau_{q+1}),
\end{align*}
and iterating for $q,q-1,q-2,\dots,1$ we obtain (define $v_t \equiv 0$ for times $t \leq 0$)
\begin{align} \label{eq:est_v1}
\sup_{t \in (1-\tau_{q},1-\tau_{q+1}]}\mathbb{E}
\|v_{t} \|_{L^2}^2
\leq
\sum_{k \leq q}
\varepsilon^{-2}_k \kappa_k (\tau_k-\tau_{k+1})
%\leq
%2\varepsilon^{-2}_q \kappa_q (\tau_q-\tau_{q+1})
\lesssim
\varepsilon^{-2}_q \kappa_q.
\end{align}
Here we are assuming $\varepsilon^{-2}_q \kappa_q \gg \varepsilon^{-2}_{q-1} \kappa_{q-1}$.
Once we have the estimate above for the $L^2$ norm of $v$ at every fixed time, we can deduce as well
\begin{align} \label{eq:est_v2}
\int_{1-\tau_{q}}^{1-\tau_{q+1}} \mathbb{E}\|v_s\|_{H^1}^2 ds \lesssim
\varepsilon^{-2}_q \kappa_q,
\end{align}
and
\begin{align} \label{eq:est_v3}
\int_{1-\tau_{q}}^{1-\tau_{q+1}} \mathbb{E}\|v_s\|_{L^2}^2 ds \lesssim
\varepsilon^{-1}_q \kappa_q.
\end{align}

Applying the It\=o Formula to $\|v^n_t\|^4_{L^2}$ we get with similar arguments
\begin{align} \label{eq:est_v4}
\int_{1-\tau_{q}}^{1-\tau_{q+1}} \mathbb{E}\|v_s\|_{L^2}^4 ds \lesssim
\varepsilon^{-2}_q \kappa_q^2.
\end{align}

\subsubsection{Estimates on $w$}
We have the explicit expression for the stochastic convolution
\begin{align*}
w_t
&=
\varepsilon^{-1/2}_q\int_{1-\tau_q}^t e^{-\varepsilon^{-1}_q(t-s)} Q_q^{1/2} dW_s
\\
&=
\sum_{k,\alpha} \theta^q_k a_{k,\alpha} e^{2 \pi ik\cdot x}
\varepsilon^{-1/2}_q\int_{1-\tau_q}^t e^{-\varepsilon^{-1}_q(t-s)} dW^{k,\alpha}_s.
\end{align*}

The estimates we need on $w$ are the following: for every $t \in (1-\tau_q,1-\tau_{q+1}]$ and $\theta \geq 0$,
\begin{align} \label{eq:est_w1}
\mathbb{E} \left\|
\varepsilon^{-1/2}_q\int_{1-\tau_q}^t e^{-\varepsilon^{-1}_q(t-s)} Q_q^{1/2} dW_s
\right\|_{H^\theta}^2
\lesssim
\kappa_q N_q^{2\theta},
\end{align}
which can be proved on the Galerkin approximations $w^n$ by It\=o isometry. By Gaussianity we also have
\begin{align} \label{eq:est_w2}
\mathbb{E} \left\|
\varepsilon^{-1/2}_q\int_{1-\tau_q}^t e^{-\varepsilon^{-1}_q(t-s)} Q_q^{1/2} dW_s
\right\|_{H^\theta}^4
\lesssim
\kappa_q^2 N_q^{4\theta}.
\end{align}

If we want to put the supremum over time inside the expectation (this will be needed in the proof of \autoref{thm:main_NS}), we can invoke \cite[Lemma 3.1]{AsFlPa21}. As a result we get a similar estimate as \eqref{eq:est_w2}, up to a logarithmic factor in $\varepsilon_q^{-1}$
\begin{align} \label{eq:est_w3}
\mathbb{E} \sup_{t \in [1-\tau_q,1-\tau_{q+1}]} \left\|
\varepsilon^{-1/2}_q\int_{1-\tau_q}^t e^{-\varepsilon^{-1}_q(t-s)} Q_q^{1/2} dW_s
\right\|_{H^\theta}^4
\lesssim
\log^2(1+\varepsilon_q^{-1})
\kappa_q^2 N_q^{4\theta}.
\end{align}

\subsubsection{Estimates on $r$}
We are left with the a priori estimates on $r$.
Let $C_\varepsilon = -Id + \varepsilon A$ and rewrite
\begin{align*}
dr
&=
-\varepsilon^{-1} r dt
+
A(\varepsilon^{-1/2} w + r) dt
+
b(\varepsilon^{-1/2} w + r,\varepsilon^{-1/2} w + r)dt
\\
&=
\varepsilon^{-1} C_\varepsilon r dt
+
\varepsilon^{-1/2} Aw dt
+
b(v,\varepsilon^{-1/2} w + r)dt.
\end{align*}
Testing the equation against the solution itself, we have the following estimate for the Galerkin truncations $r^n$, for every $t \in (1-\tau_q,1-\tau_{q+1}]$
\begin{align*}
\|r^n_{t} \|_{L^2}^2
&+
2 \int_{1-\tau_{q}}^t \|r^n_s \|_{H^1}^2 ds
+
2 \varepsilon^{-1}_q \int_{1-\tau_{q}}^t \|r^n_s \|_{L^2}^2 ds
\\
&\leq
\| v_{1-\tau_q} \|_{L^2}^2
+
2\varepsilon^{-1/2}_q
\int_{1-\tau_{q}}^t
\langle Aw_s, r^n_s \rangle ds
+
2\varepsilon^{-1}_q
\int_{1-\tau_{q}}^t
\langle b(\varepsilon^{1/2}_q v^n_s, w^n_s), r^n_s \rangle ds.
\end{align*}
By Young's inequality, there exist unimportant constants $c<1$ and $C<\infty$ such that
\begin{align*}
\|r^n_{t} \|_{L^2}^2
&+
2\int_{1-\tau_q}^t \|r^n_s \|_{H^1}^2 ds
+
2 \varepsilon^{-1}_q \int_{1-\tau_{q}}^t \|r^n_s \|_{L^2}^2 ds
\\
&\leq \nonumber
\| v_{1-\tau_q} \|_{L^2}^2
+
c \varepsilon^{-1}_q \int_{1-\tau_{q}}^t \|r^n_s \|_{L^2}^2 ds
+
C \int_{1-\tau_{q}}^t \| w^n_s \|_{H^2}^2 ds
\\
&\quad+ \nonumber
C \varepsilon^{-1}_q \int_{1-\tau_{q}}^t
\|\varepsilon^{1/2}_q v^n_s \|_{L^2}
\| w^n_s \|_{H^3}
\| r^n_s \|_{L^2} ds
\\
&\leq \nonumber
\| v_{1-\tau_q} \|_{L^2}^2
+
c \varepsilon^{-1}_q \int_{1-\tau_{q}}^t \|r^n_s \|_{L^2}^2 ds
+
C \int_{1-\tau_q}^t \| w^n_s \|_{H^2}^2 ds
\\
&\quad+ \nonumber
C \varepsilon^{-1}_q \int_{1-\tau_{q}}^t
\|\varepsilon^{1/2}_q v^n_s \|_{L^2}^4 ds
+
C \varepsilon^{-1}_q \int_{1-\tau_{q}}^t
\| w^n_s \|_{H^3}^4 ds.
\end{align*}

Taking expectations, using \eqref{eq:est_v4}, \eqref{eq:est_w2} and passing to the limit $n \to \infty$, we deduce the following preliminary estimate on $r$:
\begin{align} \label{eq:r_eps_L2}
\int_{1-\tau_q}^{1-\tau_{q+1}}
\mathbb{E}\|r_s\|_{L^2}^2 ds
\lesssim
\kappa_q^2 N_q^{12},
\end{align}
which is not very good (recall that $\kappa_q^2 N_q^{12} \to \infty$ relatively fast as $q \to \infty$) but is auxiliary to isolate the leading order terms in the dynamics of $r$.
Indeed, rewrite
\begin{align*}
dr
&=
\varepsilon^{-1} C_\varepsilon r dt
+
\varepsilon^{-1/2} Aw dt
+
b(v,\varepsilon^{-1/2} w + r)dt
\\
&=
\varepsilon^{-1} C_\varepsilon r dt
+
\varepsilon^{-1/2} Aw dt
+
\varepsilon^{-1} b( w , w )dt
+
\varepsilon^{-1/2} b( r , w )dt
+
\varepsilon^{-1/2} b( \varepsilon^{1/2} v , r )dt.
\end{align*}
Taking into account $r_{1-\tau_q}=v_{1-\tau_q}$, the mild formulation of the previous equation takes the form
\begin{align*}
	r_t
	&=
	e^{\varepsilon^{-1} C_\varepsilon (t-1+\tau_q)} v_{1-\tau_q}
	+
	\varepsilon^{-1/2} \int_{1-\tau_{q}}^t
	e^{\varepsilon^{-1} C_\varepsilon (t-s)}
	A w_s ds
	+
	\varepsilon^{-1} \int_{1-\tau_{q}}^t
	e^{\varepsilon^{-1} C_\varepsilon (t-s)}
	b(w_s,w_s) ds
	\\
	&\quad
	+
	\varepsilon^{-1/2} \int_{1-\tau_{q}}^t
	e^{\varepsilon^{-1} C_\varepsilon (t-s)}
	b(r_s,w_s) ds
	+
	\varepsilon^{-1/2} \int_{1-\tau_{q}}^t
	e^{\varepsilon^{-1} C_\varepsilon (t-s)}
	b(\varepsilon^{1/2} v_s,r_s) ds.
\end{align*}
Let $\theta_0 > 5/2$.
After taking expectation and time integral on $[1-\tau_q,1-\tau_{q+1}]$, we can separately estimate in $H^{-\theta_0}$ each term on the right-hand-side of the equation above as follows.
First,
	\begin{align*}
		\int_{1-\tau_{q}}^{1-\tau_{q+1}}
		\mathbb{E} \left\|
		e^{\varepsilon^{-1} C_\varepsilon (t-1+\tau_{q})}
		v_{1-\tau_q}\right\|_{H^{-\theta_0}} dt
		\lesssim
		\varepsilon_q \mathbb{E} \|v_{1-\tau_q}\|_{L^2}
		\lesssim
		\varepsilon_q \varepsilon_{q-1}^{-1} \kappa_{q-1}^{1/2}
		\lesssim
		\varepsilon_q^{1/2} ,
\end{align*}
where we have used \eqref{eq:est_v1} and assuming for every $q \geq 1$
\begin{align} \label{eq:condition_vareps}
\varepsilon_{q-1}^{-1} \kappa_{q-1}^{1/2} \lesssim \varepsilon_q^{-1/2}.
\end{align}

Moreover, by Young's convolution inequality and previous estimates \eqref{eq:est_v3}, \eqref{eq:est_w1} on $v$, $w$ and \eqref{eq:r_eps_L2} on $r$ we have
\begin{align*}
&\varepsilon^{-1/2}_q \int_{1-\tau_{q}}^{1-\tau_{q+1}}
\mathbb{E} \left\| \int_{1-\tau_{q}}^t
e^{\varepsilon^{-1} C_\varepsilon (t-s)}
A w_s \,ds
\right\|_{H^{-\theta_0}} dt
\\&\lesssim\varepsilon^{-1/2}_q \int_{1-\tau_{q}}^{1-\tau_{q+1}}
 \int_{1-\tau_{q}}^t
e^{-\varepsilon^{-1} (t-s)}
\mathbb{E} \| w_s\|_{L^2} \,dsdt
\lesssim
\varepsilon^{1/2}_q \kappa_q^{1/2},
\end{align*}
and similarly
\begin{align*}
\varepsilon^{-1/2}_q  \int_{1-\tau_{q}}^{1-\tau_{q+1}}
\mathbb{E} \left\| \int_{1-\tau_{q}}^t
e^{\varepsilon^{-1} C_\varepsilon (t-s)}
b(r_s, w_s) \,ds
\right\|_{H^{-\theta_0}}dt
\lesssim
\varepsilon^{1/2}_q \kappa_q^{3/2} N_q^6,
\end{align*}
\begin{align*}
\varepsilon^{-1/2}_q  \int_{1-\tau_{q}}^{1-\tau_{q+1}}
\mathbb{E} \left\| \int_{1-\tau_{q}}^t
e^{\varepsilon^{-1} C_\varepsilon (t-s)}
b(\varepsilon^{1/2}_q v_s, r_s) \,ds
\right\|_{H^{-\theta_0}} dt
\lesssim
\varepsilon^{1/2}_q \kappa_q^{3/2} N_q^6.
\end{align*}
In the second and last inequality we have used that the operator $b:H \times H \to H^{-\theta_0}$ is bounded, therefore
\begin{align*}
\mathbb{E}  \left\|b(r_s, w_s)\right\|_{H^{-\theta_0}}
\lesssim
\mathbb{E}  \|r_s\|_{L^2} \|w_s \|_{L^2}
\leq
\left(\mathbb{E}  \|r_s\|_{L^2}^2\right)^{1/2}
\left(\mathbb{E}  \|w_s \|_{L^2}^2\right)^{1/2},
\end{align*}
and similarly for the term $b(\varepsilon^{1/2}_q v_s, r_s)$.
Putting all together, we obtain
\begin{align} \label{eq:est_r1}
\int_{1-\tau_q}^{1-\tau_{q+1}} \mathbb{E} \|r_s - \tilde{r}_s \|_{H^{-\theta_0}} ds
\lesssim
\varepsilon^{1/2}_q \kappa_q^{3/2} N_q^6,
\end{align}
where we have defined
$$\tilde{r}_t = \varepsilon_{q}^{-1} \int_{1-\tau_{q}}^t
	e^{\varepsilon^{-1} C_\varepsilon (t-s)}
	b(w_s,w_s) ds.$$
An estimate similar to \eqref{eq:r_eps_L2} holds for $\tilde{r}$, indeed
\begin{align*}
\int_{1-\tau_q}^{1-\tau_{q+1}} \mathbb{E} \| \tilde{r}_s \|_{L^2}^2 ds
&\lesssim
\int_{1-\tau_q}^{1-\tau_{q+1}} \int_{1-\tau_q}^s \varepsilon_q^{-1} e^{-\varepsilon^{-1} (s-r)}\mathbb{E} \| b(w_r,w_r) \|_{L^2}^2  drds
\\
&\lesssim
\int_{1-\tau_q}^{1-\tau_{q+1}} \int_{1-\tau_q}^s \varepsilon_q^{-1} e^{-\varepsilon^{-1} (s-r)}\mathbb{E} \| w_r \|_{H^{3}}^4  drds
\lesssim
\kappa_q^2 N_q^{12},
\end{align*}
and more generally, using $b:H^{3+\theta} \times H^{3+\theta} \to H^\theta$ continuously for every $\theta \geq 0$,
\begin{align} \label{eq:tilde_r}
\int_{1-\tau_q}^{1-\tau_{q+1}} \mathbb{E} \| \tilde{r}_s \|_{H^\theta}^2 ds
&\lesssim
\kappa_q^2 N_q^{12+4\theta}.
\end{align}
By interpolation we also get for every $\theta \in (0,\theta_0)$ and $p$ such that $\frac{p\theta}{\theta_0}+\frac{p(1-\theta/\theta_0)}{2}\leq 1$
\begin{align*}
\int_{1-\tau_q}^{1-\tau_{q+1}} \mathbb{E} \|r_s- \tilde{r}_s \|_{H^{-\theta}}^p ds
&\lesssim
\int_{1-\tau_q}^{1-\tau_{q+1}} \mathbb{E}
\|r_s- \tilde{r}_s \|_{H^{-\theta_0}}^{p\theta/\theta_0}
\|r_s- \tilde{r}_s \|_{L^2}^{p(1-\theta/\theta_0)}  ds
\\
&\lesssim
\left( \varepsilon^{1/2}_q \kappa_q^{3/2} N_q^6 \right)^{p\theta/\theta_0}
\left( \kappa_q^2 N_q^{12} \right)^{p(1-\theta/\theta_0)/2}.
\end{align*}
In particular, for $\theta=1/2$ and $p=4/3$
\begin{align} \label{eq:est_r2}
\left(\int_{1-\tau_q}^{1-\tau_{q+1}} \mathbb{E}\|r_s- \tilde{r}_s\|_{H^{-1/2}}^{4/3} ds \right)^{3/4}
&\lesssim
\left( \varepsilon^{1/2}_q \kappa_q^{3/2} N_q^6 \right)^{1/2\theta_0}
\left( \kappa_q^2 N_q^{12} \right)^{(1-1/2\theta_0)/2}
\lesssim
\varepsilon^{1/12}_q
\kappa_q^{2}
N_q^{6}.
\end{align}

\subsection{Dissipation}
As we have seen in \autoref{sec:white_noise}, when $v$ is a white-in-time noise we can prove anomalous dissipation for $u$ thanks to the presence of a Stratonovich corrector in the It\=o formulation of \eqref{eq:NS_transport}.
The goal of this subsection is to ``find'' the hidden Stratonovich corrector in the dynamics of $u$ solution of \eqref{eq:NS_intro}. We will focus on the Navier-Stokes case \eqref{eq:NS_intro} only, and in particular on \autoref{thm:main_NS}; but the same arguments work with minor modifications in all the other cases.

Technically speaking, there isn't any Stratonovich corrector in \eqref{eq:NS_intro} since $u$ has finite variation; but morally speaking, the solution $v$ of \eqref{eq:veps3} excites the small scales of $u$ just as well as the white-in-time transport noise.
Therefore, it is reasonable to expect that the same dissipation mechanism, induced by transfer of energy to high wavenumbers, can happen in this case, too.

Let us work on a fixed time interval $(1-\tau_q,1-\tau_{q+1}]$.
Let us consider the process
\begin{align*}
U
:=
u
+
\varepsilon^{1/2}_q b(w,u)
+
\frac{\varepsilon_q}{2} b(w,b(w,u))
+V,
\end{align*}
where the two correctors $\varepsilon^{1/2}_q b(w,u)$ and $\frac{\varepsilon_q}{2} b(w,b(w,u))$ are motivated by the heuristic arguments presented in \autoref{ssec:rough} and serve to reintroduce the time roughness producing the Stratonovich corrector in \eqref{eq:NS_intro}, and the auxiliary process $V$ is defined as
\begin{align*}
V
:=
\frac{\varepsilon_q}{2} b(b(w,w),u)
+
\varepsilon_q b((-C_\varepsilon)^{-1} \tilde{r},u),
\end{align*}
and is needed to compensate for the term $b(r,u)dt$ appearing in the dynamics of $u$.
Indeed, by It\=o formula we have
\begin{align*}
\frac{\varepsilon_q}{2}  d \left( b(b(w,w),u)\right)
&=
-b(b(w,w),u) \,dt
+
\sum_{k,\alpha} (\theta_k^q)^2 b(b(a_{k,\alpha}e_k,a_{k,\alpha}e_{-k}),u) \,dt
\\
&\quad
+\frac{\varepsilon_q^{1/2}}{2} b(b(Q_q^{1/2}dW_t,w),u)
+\frac{\varepsilon_q^{1/2}}{2} b(b(w,Q_q^{1/2}dW_t),u)
\\
&\quad+
\frac{\varepsilon_q^{1/2}}{2} b(b(w,w),b(w,u)) \,dt
+
\frac{\varepsilon_q}{2} b(b(w,w),\nu Au + b(u,u)+b(r,u))\,dt
\end{align*}
and
\begin{align*}
\varepsilon_q d \left( b((-C_\varepsilon)^{-1} \tilde{r},u)\right)
&=
-b(\tilde{r},u) \,dt+b((-C_\varepsilon)^{-1} b(w,w),u)\,dt
\\
&\quad+
\varepsilon_q b((-C_\varepsilon)^{-1} \tilde{r},\nu Au + b(u,u)+b(r,u))\,dt
\\
&\quad+
\varepsilon^{1/2}_q b((-C_\varepsilon)^{-1} \tilde{r},b(w,u))\,dt.
\end{align*}
Notice that the It\=o corrector $\sum_{k,\alpha} (\theta_k^q)^2 b(b(a_{k,\alpha}e_k,a_{k,\alpha}e_{-k}),u)$ in the dynamics of  $\frac{\varepsilon_q}{2} b(b(w,w),u)$ equals zero since $b(a_{k,\alpha}e_k,a_{k,\alpha}e_{-k})=0$ for every $k \in \Z^3_0$ and $\alpha \in \{1,2\}$.
Moreover, since it holds $(-C_\varepsilon)^{-1} - Id = \varepsilon_q A (-C_\varepsilon)^{-1}$ (this can be checked multiplying both expression by $-C_\varepsilon$) we also have
\begin{align*}
b((-C_\varepsilon)^{-1} b(w,w),u) - b(b(w,w),u)
&=
\varepsilon_q b(A (-C_\varepsilon)^{-1} b(w,w),u).
\end{align*}

%The key idea is that the small perturbation $\varepsilon^{1/2}_q b(w,u)+\frac{\varepsilon_q}{2} b(w,b(w,u))$ reintroduces exactly the time roughness producing the Stratonovich corrector.
%%To see this, let us denote $\mu$ the invariant measure of the process $d\tilde{w} = -\tilde{w} dt + Q_q^{1/2} d\tilde{W}_t$.
%Indeed, by It\=o formula the process $U$ evolves according to

Therefore, the process $U$ evolves according to
\begin{align} \label{eq:dU}
dU
&=
\nu A u \,dt
+
b(u,u) \,dt
+
b(r-\tilde{r},u) \,dt
+
\sum_{k,\alpha} (\theta_k^q)^2 b(a_{k,\alpha}e_k,b(a_{k,\alpha}e_{-k},u))
\,dt
+
b(Q^{1/2}_q dW_t, u)
\\
&\quad+ \nonumber
\varepsilon^{1/2}_q b(w,\nu A u
+
b(u,u)
+
b(r,u)) \,dt
+
\frac{\varepsilon^{1/2}_q}{2} b(w,b(w,b(w,u))  \,dt
\\
&\quad+ \nonumber
\frac{\varepsilon_q}{2} b(w,b(w,\nu Au+b(u,u)+b(r,u))) \,dt
\\
&\quad+ \nonumber
\frac{\varepsilon^{1/2}_q}{2} b(Q^{1/2}_q dW_t,b(w,u))
+
\frac{\varepsilon^{1/2}_q}{2} b(w,b(Q^{1/2}_q dW_t,u))
\\
&\quad \nonumber
+\frac{\varepsilon_q^{1/2}}{2} b(b(Q_q^{1/2}dW_t,w),u)
+\frac{\varepsilon_q^{1/2}}{2} b(b(w,Q_q^{1/2}dW_t),u)
\\
&\quad \nonumber
+\frac{\varepsilon_q^{1/2}}{2} b(b(w,w),b(w,u)) \,dt
+
\frac{\varepsilon_q}{2} b(b(w,w),\nu Au + b(u,u)+b(r,u))\,dt
\\
&\quad \nonumber
+\varepsilon_q b((-C_\varepsilon)^{-1} \tilde{r},\nu Au + b(u,u)+b(r,u))\,dt
\\
&\quad \nonumber
+\varepsilon^{1/2}_q b((-C_\varepsilon)^{-1} \tilde{r},b(w,u))\,dt
+\varepsilon_q b(A (-C_\varepsilon)^{-1} b(w,w),u) \,dt
.
\end{align}

The term $\sum_{k,\alpha} (\theta_k^q)^2 b(a_{k,\alpha}e_k,b(a_{k,\alpha}e_{-k},u))$ comes from the second derivative of $\frac{\varepsilon_q}{2} b(w,b(w,u))$ with respect to $w$, and recalling \eqref{eq:defin_strat} it coincides with the Stratonovich-to-It\=o corrector applied to $u$:
\begin{align*}
\sum_{k,\alpha} (\theta_k^q)^2 b(a_{k,\alpha}e_k,b(a_{k,\alpha}e_{-k},u))
=
S_q(u).
\end{align*}

Thus, we can rewrite
\begin{align} \label{eq:Squ}
S_q(u)
&=
\frac25 \kappa_q \Delta U
+
\left( S_q(U) - \frac25 \kappa_q \Delta U \right)
-
\varepsilon^{1/2}_q
S_q(b(w,u)
-
\frac{\varepsilon_q}{2}
S_q(b(w,b(w,u)))
-
S_q(V),
\end{align}
and the term $\frac25 \kappa_q \Delta U$ gives us enough dissipation to control a negative Sobolev norm in the mild formulation of $U$.
More precisely, let $\Pi_L$ be the Fourier projector onto modes $|k| \leq N_q^{1-\delta}$, for some $\delta \in (0,1)$, and denote $U^L = \Pi_L U$.
We have
\begin{lem} \label{lem:h-4}
Let $u$ be a Leray-Hopf solution to \eqref{eq:NS_intro} with zero-mean, divergence-free initial condition $u_0$ satisfying $\|u_0\|_{L^2}\leq M$ for some deterministic $1\leq M < \infty$, and let $U^L$ be defined as above.
Then for every $q \in \N$ there exists a choice of the parameters $\tau_q$, $N_q$, $\varepsilon_q$ and $\kappa_q \sim N_q^3$ such that for every $\delta \in (1/6,1)$
\begin{align*}
\mathbb{E}\left[\sup_{t \in [1-\tau_q/2,1-\tau_{q+1}]}
\|U^L\|_{H^{-4}} \right]
\lesssim_{\delta}
\frac{M^2}{(\nu+\kappa_q)^{1-\delta}}
+
\frac{\varepsilon_q^{1/12}\kappa_q^2 N_q^6 M}{\nu^{1/4}}.
\end{align*}
\end{lem}

\begin{proof}
Let $P$ be the semigroup generated by $\nu A + 2\kappa_q \Delta/5 = (\nu +2\kappa_q/5)\Delta$, and consider the mild formulation of \eqref{eq:dU} for times $t \in [1-\tau_q/2,1-\tau_{q+1}]$, taking also \eqref{eq:Squ} into account.

First, at time $t=1-\tau_q$ we have $U^L=u^L$ by definition of $w$, therefore
\begin{align*}
\mathbb{E}\left[ \sup_{t \in [1-\tau_q/2,1-\tau_{q+1}]}
\|P(t-(1-\tau_q)) U^L_{1-\tau_q}\|_{H^{-4}}
\right]
\lesssim
\frac{M}{(\nu+\kappa_q)\tau_q}.
\end{align*}
We shall assume hereafter condition \eqref{eq:condition_parameters} on $\tau_q$ and $\kappa_q$ as in \autoref{lem:h-1}, and moreover $(\nu +\kappa_q)^{-\delta} \leq \tau_q$.
In addition, we will take $\varepsilon_q$ satisfying \eqref{eq:condition_vareps} and small with respect to the other parameters, so to control easily all the terms multiplied by powers of $\varepsilon_q$.
For these terms, we do not need the action of the semigroup $P$ to prove smallness, and we will use the simple estimate
$$
\mathbb{E}\left[\sup_{t \in [1-\tau_q/2,1-\tau_{q+1}]} \left\| \int_{1-\tau_q}^t P(t-s) \Pi_L\dots ds \right\|_{H^{-4}} \right]
\lesssim
\int_{1-\tau_q}^{1-\tau_{q+1}} \mathbb{E}\left\| \dots \right\|_{H^{-4}} ds.
$$
Having said this, let us control the terms appearing in \eqref{eq:Squ}, which permits us to isolate the strong dissipation term $\frac25 \kappa_q \Delta U^L$.
By \eqref{eq:est_w1}, \eqref{eq:r_eps_L2}, using $\|S_q(u)\|_{H^{-4}}\lesssim \kappa_q\|u\|_{H^{-2}}$ and the Sobolev embedding we have for $\theta_0>5/2$
\begin{align*}&
\varepsilon_q^{1/2}
\mathbb{E}\left[
\int_{1-\tau_q}^{1-\tau_{q+1}} \left\|
 S_q(b(w_s,u_s)) \right\|_{H^{-4}}ds
\right]\lesssim \varepsilon_q^{1/2}\kappa_q
\mathbb{E}\left[
\int_{1-\tau_q}^{1-\tau_{q+1}}
 \|w_s\otimes u_s \|_{H^{-1}}ds
\right]
\\&\lesssim\varepsilon_q^{1/2}\kappa_q
\int_{1-\tau_q}^{1-\tau_{q+1}}
 \mathbb{E}\|w_s\|_{L^\infty}\| u_s \|_{L^2}ds
\lesssim
\varepsilon_q^{1/2}
\kappa_q^{3/2} N_q^{\theta_0-1} M,
\end{align*}
\begin{align*}&
\varepsilon_q
\mathbb{E}\left[
\int_{1-\tau_q}^{1-\tau_{q+1}} \left\|
 S_q(b(w_s,b(w_s, u_s))) \right\|_{H^{-4}}ds
\right]\lesssim \varepsilon_q\kappa_q
\mathbb{E}\left[
\int_{1-\tau_q}^{1-\tau_{q+1}}
 \|w_s\otimes b(w_s,u_s) \|_{H^{-1}}ds
\right]
\\&\lesssim\varepsilon_q\kappa_q
\int_{1-\tau_q}^{1-\tau_{q+1}}
 \mathbb{E}\|w_s\|_{H^{\theta_0}}^2\| u_s \|_{L^2}ds
\lesssim
\varepsilon_q
\kappa_q^{2} N_q^{2\theta_0} M,
\end{align*}
and similarly by \eqref{eq:tilde_r}
\begin{align*}
\varepsilon_q
\mathbb{E}\left[
\int_{1-\tau_q}^{1-\tau_{q+1}} \left\| S_q(b(b(w_s,w_s), u_s)))
\right\|_{H^{-4}} ds \right]
\lesssim
\varepsilon_q
\kappa_q^{2} N_q^{2\theta_0} M,
\end{align*}
\begin{align*}
\varepsilon_q
\mathbb{E}\left[
\int_{1-\tau_q}^{1-\tau_{q+1}} \left\| S_q(b((-C_\varepsilon)^{-1}\tilde{r}_s, u_s))
\right\|_{H^{-4}} ds \right]
\lesssim \varepsilon_q
\kappa_q^2 N_q^{4+2\theta_0}M.
\end{align*}

Moreover, by \eqref{eq:est_w2} and \eqref{eq:tilde_r} the correctors giving the difference $U-u$ are small, thus
\begin{align*}
\mathbb{E}&\left[\sup_{t \in [1-\tau_q/2,1-\tau_{q+1}]}
\left\|
\int_{1-\tau_q}^t
P(t-s)
\left(
S_q(U^L) - \frac25 \kappa_q \Delta U^L
\right) ds
\right\|_{H^{-4}}
\right]
\\
&\lesssim
\kappa_q N_q^{-\delta}
\mathbb{E}\left[\sup_{t \in [1-\tau_q/2,1-\tau_{q+1}]}
\int_{1-\tau_q}^t
\left\| P(t-s)U^L \right\|_{H^{-2}}ds
\right]
\\
&\lesssim
{\kappa_q^\varepsilon} N_q^{-\delta} M\left( 1 + \varepsilon_q^{1/2} \kappa_q^{1/2}N_q^{\theta_0-1} + \varepsilon_q \kappa_q N_q^{4+2\theta_0} \right),
\end{align*}
for arbitrary $\varepsilon>0$ small, coming from the action of the semigroup.

Similarly we have
\begin{align*}
\mathbb{E}&\left[\sup_{t \in [1-\tau_q/2,1-\tau_{q+1}]}
\left\|
\int_{1-\tau_q}^t
P(t-s)\Pi_L
\nu A (u_s-U_s)
 ds
\right\|_{H^{-4}}
\right]
\\
&\lesssim M\left( \varepsilon_q^{1/2} \kappa_q^{1/2}N_q^{\theta_0} + \varepsilon_q \kappa_q N_q^{6+2\theta_0} \right)
\end{align*}

Let us now move to the other terms in \eqref{eq:dU}. We  use \eqref{eq:est_w2}, \eqref{eq:r_eps_L2} and
have
\begin{align*}
\varepsilon_q^{1/2}
\mathbb{E}&\left[
\int_{1-\tau_q}^{1-\tau_{q+1}} \left\|
b(w_s,\nu A u_s
+
b(u_s,u_s)+
b(r_s,u_s)
)\right\|_{H^{-4}}
ds
\right]
\\
&\lesssim
\varepsilon_q^{1/2}
\mathbb{E}\left[
\int_{1-\tau_q}^{1-\tau_{q+1}}
( \|w_s\|_{H^{\theta_0+1}}\|u_s \|_{L^2}+\|w_s\|_{H^{\theta_0}}(\|u_s\|^2_{L^2}+\|u_s\|_{L^2}\|r_s\|_{L^2})ds
\right]
\\
&\lesssim
\varepsilon_q^{1/2}\kappa_q^{3/2} N_q^{6+\theta_0}
M^2,
\end{align*}
\begin{align*}
\varepsilon_q
\mathbb{E}&\left[
\int_{1-\tau_q}^{1-\tau_{q+1}} \left\|
b(w_s,b(w_s,\nu Au_s
+b(u_s,u_s)+b(r_s,u_s)))\right\|_{H^{-4}}
ds
 \right]
\\
&\lesssim
\varepsilon_q
\mathbb{E}\left[
\int_{1-\tau_q}^{1-\tau_{q+1}}
( \|w_s\|_{H^{\theta_0+1}} \|w_s\|_{H^{\theta_0+2}}\|u_s \|_{L^2}+\|w_s\|_{H^{\theta_0}}\|w_s\|_{H^{\theta_0+1}}(\|u_s\|^2_{L^2}+\|u_s\|_{L^2}\|r_s\|_{L^2}))ds
\right]
\\
&\lesssim
\varepsilon_q
\kappa_q^2 N_q^{2\theta_0+7}M^2,
\end{align*}
and similarly by \eqref{eq:tilde_r}
\begin{align*}
\varepsilon_q
\mathbb{E}\left[
\int_{1-\tau_q}^{1-\tau_{q+1}} \left\|
b(b(w_s,w_s),\nu Au_s + b(u_s,u_s)+b(r_s,u_s)) \right\|_{H^{-4}} ds
\right]
\lesssim
\varepsilon_q \kappa_q^2 N_q^{2\theta_0+8}
M^2,
\end{align*}
\begin{align*}
\varepsilon_q
\mathbb{E}\left[
\int_{1-\tau_q}^{1-\tau_{q+1}} \left\|
b((-C_\varepsilon)^{-1} \tilde{r}_s,\nu Au_s + b(u_s,u_s)+b(r_s,u_s)) \right\|_{H^{-4}} ds
\right]
\lesssim
\varepsilon_q \kappa_q^2 N_q^{14+2\theta_0}
M^2.
\end{align*}
In addition,
\begin{align*}
\varepsilon_q^{1/2}
\mathbb{E}&\left[
\int_{1-\tau_q}^{1-\tau_{q+1}} \left\|
b(w_s,b(w_s,
b(w_s,u_s)) \right\|_{H^{-4}} ds
\right]
\\&\lesssim \varepsilon_q^{1/2}
\mathbb{E}\left[
\int_{1-\tau_q}^{1-\tau_{q+1}}
 \|w_s\|_{H^{\theta_0+1}} \|w_s\|_{H^{\theta_0}}\|w_s\|_{H^{\theta_0-1}}\|u_s \|_{L^2}ds
\right]\\
&\lesssim
\varepsilon_q^{1/2}
\kappa_q^{3/2} N_q^{3\theta_0} M,
\end{align*}
and
\begin{align*}
\varepsilon_q^{1/2}
\mathbb{E}\left[
\int_{1-\tau_q}^{1-\tau_{q+1}} \left\|
b(b(w_s,w_s),b(w_s,u_s)) \right\|_{H^{-4}} ds
\right]
\lesssim
\varepsilon_q^{1/2}
\kappa_q^{3/2} N_q^{3\theta_0} M,
\end{align*}
\begin{align*}
\varepsilon^{1/2}_q
\mathbb{E}\left[
\int_{1-\tau_q}^{1-\tau_{q+1}} \left\|
b((-C_\varepsilon)^{-1} \tilde{r}_s,b(w_s,u_s)) \right\|_{H^{-4}} ds
\right]
\lesssim
\varepsilon_q^{1/2}
\kappa_q^{3/2} N_q^{6+3\theta_0} M,
\end{align*}
\begin{align*}
\varepsilon_q
\mathbb{E}\left[
\int_{1-\tau_q}^{1-\tau_{q+1}} \left\|
b(A (-C_\varepsilon)^{-1} b(w_s,w_s),u_s) \right\|_{H^{-4}} ds
\right]
\lesssim
\varepsilon_q
\kappa_q N_q^{2\theta_0} M.
\end{align*}
All these quantities are small assuming
$$\varepsilon_q \kappa_q^3 N_q^{30} \leq 1. $$

For the terms involving stochastic integrals, we have by maximal inequality for stochastic convolution \cite[Theorem 1.1]{DaPZ92}
\begin{align*}
\varepsilon_q
\mathbb{E}&\left[\sup_{t \in [1-\tau_q/2,1-\tau_{q+1}]}
\left\|
\int_{1-\tau_q}^t
P(t-s)\Pi_L
b(Q_q^{1/2} dW_s,b(w_s,u_s))
\right\|_{H^{-4}}^2 \right]
\\
&\lesssim
\varepsilon_q \sum_{k,\alpha}
\int_{1-\tau_q}^{1-\tau_{q+1}}
\mathbb{E}\left[
\left\|
b(\theta_k^q a_{k,\alpha} e_k,b(w_s,u_s))
\right\|_{H^{-4}}^2  \right]ds
\lesssim
\varepsilon_q \kappa_q^2 N_q^{4\theta_0} M^2,
\end{align*}
and in the same fashion
\begin{align*}
\varepsilon_q
\mathbb{E}\left[\sup_{t \in [1-\tau_q/2,1-\tau_{q+1}]}
\left\|
\int_{1-\tau_q}^t
P(t-s)\Pi_L
b(w_s,b(Q_q^{1/2} dW_s,u_s))
\right\|_{H^{-4}}^2 \right]
\lesssim
\varepsilon_q \kappa_q^2 N_q^{4\theta_0} M^2,
\end{align*}
\begin{align*}
\varepsilon_q
\mathbb{E}\left[\sup_{t \in [1-\tau_q/2,1-\tau_{q+1}]}
\left\|
\int_{1-\tau_q}^t
P(t-s)\Pi_L
b(b(Q_q^{1/2} dW_s,w_s),u_s)
\right\|_{H^{-4}}^2 \right]
\lesssim
\varepsilon_q \kappa_q^2 N_q^{4\theta_0} M^2,
\end{align*}
\begin{align*}
\varepsilon_q
\mathbb{E}\left[\sup_{t \in [1-\tau_q/2,1-\tau_{q+1}]}
\left\|
\int_{1-\tau_q}^t
P(t-s)\Pi_L
b(b(w_s,Q_q^{1/2} dW_s),u_s)
\right\|_{H^{-4}}^2 \right]
\lesssim
\varepsilon_q \kappa_q^2 N_q^{4\theta_0} M^2,
\end{align*}
which are small under the same assumptions on $\varepsilon_q$.

As for the other terms, we have
\begin{align*}
\mathbb{E}&\left[\sup_{t \in [1-\tau_q/2,1-\tau_{q+1}]}
\left\|
\int_{1-\tau_q}^t
P(t-s)\Pi_L
b(u_s,u_s)ds
\right\|_{H^{-4}}
\right]
\\
&\lesssim
\mathbb{E}\left[\sup_{t \in [1-\tau_q/2,1-\tau_{q+1}]}
\int_{1-\tau_q}^t
\frac{ \|u_s\|_{L^2}^2 }{(\nu + \kappa_q)^{1-\delta} (t-s)^{1-\delta}}
ds \right]
\lesssim
\frac{M^2 }{(\nu + \kappa_q)^{1-\delta}},
\end{align*}
and recalling \eqref{eq:est_r2}
\begin{align*}
\mathbb{E}&\left[\sup_{t \in [1-\tau_q/2,1-\tau_{q+1}]}
\left\|
\int_{1-\tau_q}^t
P(t-s)\Pi_L
b(r_s-\tilde{r}_s,u_s)ds
\right\|_{H^{-4}}
\right]
\\
&\lesssim
\mathbb{E}\left[
\int_{1-\tau_q}^{1-\tau_{q+1}}
\|u_s\|_{H^{1/2}} \|r_s-\tilde{r}_s\|_{H^{-1/2}} ds \right]
\\
&\lesssim
\left(
\int_{1-\tau_q}^{1-\tau_{q+1}}
\mathbb{E}\|u_s\|_{H^{1/2}}^4 ds \right)^{1/4}
\left(
\int_{1-\tau_q}^{1-\tau_{q+1}}
\mathbb{E}\|r_s-\tilde{r}_s\|_{H^{-1/2}}^{4/3} ds \right)^{3/4}
\lesssim
\frac{\varepsilon_q^{1/12}\kappa_q^2 N_q^6 M}{\nu^{1/4}} .
\end{align*}
Here we additionally need to ask
\begin{align} \label{eq:condition_varesp}
\varepsilon_q \kappa_q^{24} N_q^{72} \leq 1.
\end{align}

The last It\=o integral was already controlled in \autoref{lem:h-1} by \cite[Lemma 2.5]{FlGaLu21+} as
\begin{align*}
\mathbb{E}&\left[\sup_{t \in [1-\tau_q/2,1-\tau_{q+1}]}
\left\|
\int_{1-\tau_q}^t
P(t-s)\Pi_L
b(Q_q^{1/2} dW_s,u_s)
\right\|_{H^{-4}}^2 \right]
\lesssim
\frac{M^2}{(\nu+\kappa_q)^{1-\delta}}.
\end{align*}

Putting all together, we get for our choice of parameters \eqref{eq:condition_parameters} and \eqref{eq:condition_varesp}
\begin{align*}
\mathbb{E}\left[ \sup_{t \in [1-\tau_q/2,1-\tau_{q+1}]}
\left\| U^L_t \right\|_{H^{-4}}\right]
\lesssim
\frac{M^2}{(\nu+\kappa_q)^{1-\delta}}
+
\frac{\varepsilon_q^{1/12}\kappa_q^2 N_q^6 M}{\nu^{1/4}}.
\end{align*}
\end{proof}

We are finally ready to give the proof of our main result.
\begin{proof}[Proof of \autoref{thm:main_NS}]
The proof is similar to that of \autoref{prop:dissipation_transport} in the previous section.

First of all, by \autoref{lem:h-4}, \eqref{eq:est_w3} and condition \eqref{eq:condition_varesp} it holds
\begin{align*}
\mathbb{E} \left[\sup_{t \in [1-\tau_q/2,1-\tau_{q+1}]}
\left\| u^L_t \right\|_{H^{-4}}\right]
&\lesssim
\mathbb{E}\left[ \sup_{t \in [1-\tau_q/2,1-\tau_{q+1}]}
\left\| U^L_t \right\|_{H^{-4}}\right]
+
o(M \varepsilon_q^{1/4})
\\
&\leq
C_\delta
\left(
\frac{M^2}{(\nu+\kappa_q)^{1-\delta}}
+
\frac{\varepsilon_q^{1/12}\kappa_q^2 N_q^6 M}{\nu^{1/4}}
\right) =: M\tilde{c}_q^2.
\end{align*}
and therefore with probability at least $1-\tilde{c}_q$ we have for every $t \in [1-\tau_q/2,1-\tau_{q+1}]$
\begin{align*}
\|u_t\|_{H^{-1}}^2
\leq
M^{3/2} \|u_t\|_{H^{-4}}^{1/2}
\leq
\left( \tilde{c}_q^2 + N_q^{8(\delta-1)} \right)^{1/4}M^2  =: c_q M^2.
\end{align*}

We take $\delta$, $\tau_q$, $\kappa_q$ and $N_q$ as in \autoref{prop:dissipation_transport}, and $\varepsilon_q$ satisfying \eqref{eq:condition_vareps} and \eqref{eq:condition_varesp}, for instance $\varepsilon_q = 2^{-4^q} \kappa_q^{-36} N_q^{-72} $.
From now on the proof goes exactly as that of \autoref{prop:dissipation_transport}, and we omit it.
\end{proof}

As already mentioned, the proofs of \autoref{thm:main_passive} and \autoref{thm:OU} descend easily by the same arguments presented above.
Finally, let us give the:
\begin{proof}[Proof of \autoref{cor:passive}]
Let $\PP_{\rho_0}$ and $\PP_{\nu}$ be as in the statement of the corollary. Without loss of generality we may assume $\|\rho_0\|_{L^2} \leq M$ for $\PP_{\rho_0}$ almost every $\rho_0$ and, possibly replacing $\PP_{\nu}$ with an equivalent measure, $\nu^{-1/4}$ integrable with respect to $\PP_{\nu}$.

Let us denote $\tilde{\PP} := \PP_{\rho_0} \otimes \PP_\nu \otimes \PP$, with expectation $\tilde{\mathbb{E}}$.
For every triple $\tilde{\omega}=(\rho_0,\nu,\omega)$ there exists a unique solution $\rho=\rho(\tilde{\omega})$ of \eqref{eq:passive_scalar_intro} which satisfies, by the same computations of \autoref{lem:h-4}:
\begin{align*}
\tilde{\mathbb{E}}\left[\sup_{t \in [1-\tau_q/2,1-\tau_{q+1}]}
\|\rho^L\|_{H^{-4}} \right]
\lesssim
\frac{M^2}{\kappa_q^{1/5}}
+
\varepsilon_q^{1/12}\kappa_q^2 N_q^6 M.
\end{align*}
This implies, arguing as in the proof of \autoref{prop:dissipation_transport}
\begin{align*}
\tilde{\PP}(A_q) \geq 1-\tilde{c}_q,
\quad
A_q = \left\{ E(1-\tau_{q+1}) \leq M^2 \frac{2 c_q}{\nu \tau_q} \right\}
\end{align*}
for every $q \in \N$ and suitable $c_q$, $\tilde{c}_q$, decreasing fast enough so that Borel-Cantelli Lemma gives
\begin{align*}
\lim_{t \uparrow 1} \| \rho_t \|_{L^2} = 0
\quad
\tilde{\PP}-\mbox{almost surely}.
\end{align*}
In particular, by Fubini Theorem we have a full $\PP$-probablity set $\Omega_0 = \Omega_0(\PP_{\rho_0},\PP_\nu) \subset \Omega$ such that for every $\omega \in \Omega_0$ it holds $\lim_{t \uparrow 1} \| \rho_t \|_{L^2} = 0$ for $\PP_{\rho_0} \otimes \PP_\nu$ almost every $(\rho_0,\nu)$.
In particular, total dissipation for almost every initial condition and viscosity occurs for a generic realization $v=v(\omega)$, $\omega \in \Omega_0$.
It is interesting to observe that $\PP(\Omega_0)=1$.
\end{proof}

%\section*{Statements}
%The authors declare no conflict of interest.
%Data Availability Statement Data sharing is not applicable to this article as no data sets were generated or analyzed during the current study.

\bibliographystyle{plain}

\end{document}